\documentclass[11pt]{amsart}

\usepackage{amsmath, amsthm, amssymb, amsfonts, enumerate}
\usepackage[colorlinks=true,linkcolor=blue,urlcolor=blue]{hyperref}
\usepackage{dsfont}
\usepackage{color}
\usepackage{geometry}
\usepackage{todonotes}
\usepackage{epstopdf}

\newtheorem{theorem}{Theorem}[section]
\newtheorem{remark}[theorem]{Remark}

\newtheorem{lemma}[theorem]{Lemma}
\newtheorem{proposition}[theorem]{Proposition}
\newtheorem{corollary}[theorem]{Corollary}

\def \cL{{\mathcal L}}

\DeclareMathOperator*{\I}{Int}

\title[Optimal Extraction with Price Impact]{An Optimal Extraction Problem with Price Impact}
\author[Ferrari, Koch]{Giorgio Ferrari, Torben Koch}
\address{G.~Ferrari: Center for Mathematical Economics (IMW), Bielefeld University, Universit\"atsstrasse 25, 33615, Bielefeld, Germany}
\email{\href{mailto:giorgio.ferrari@uni-bielefeld.de}{giorgio.ferrari@uni-bielefeld.de}}
\address{T.~Koch: Center for Mathematical Economics (IMW), Bielefeld University, Universit\"atsstrasse 25, 33615, Bielefeld, Germany}
\email{\href{mailto:t.koch@uni-bielefeld.de}{t.koch@uni-bielefeld.de}}

\date{\today}

\numberwithin{equation}{section}

\begin{document}

\begin{abstract} 
A price-maker company extracts an exhaustible commodity from a reservoir, and sells it instantaneously in the spot market. In absence of any actions of the company, the commodity's spot price evolves either as a drifted Brownian motion or as an Ornstein-Uhlenbeck process. While extracting, the company affects the market price of the commodity, and its actions have an impact on the dynamics of the commodity's spot price. The company aims at maximizing the total expected profits from selling the commodity, net of the total expected proportional costs of extraction. We model this problem as a two-dimensional degenerate singular stochastic control problem with finite fuel. To determine its solution, we construct an explicit solution to the associated Hamilton-Jacobi-Bellman equation, and then verify its actual optimality through a verification theorem. On the one hand, when the (uncontrolled) price is a drifted Brownian motion, it is optimal to extract whenever the current price level is larger or equal than an endogenously determined constant threshold. On the other hand, when the (uncontrolled) price evolves as an Ornstein-Uhlenbeck process, we show that the optimal extraction rule is triggered by a curve depending on the current level of the reservoir. Such a curve is a strictly decreasing $C^{\infty}$-function for which we are able to provide an explicit expression. Finally, our study is complemented by a theoretical and numerical analysis of the dependency of the optimal extraction strategy and value function on the model's parameters. 
\end{abstract}

\maketitle

\smallskip


\smallskip

{\textbf{MSC2010 subject classification}}: 93E20; 49L20; 91B70; 91B76; 60G40. 

\smallskip

{\textbf{OR/MS subject classification}}: Dynamic programming/optimal control: applications, Markov; Probability: stochastic models applications, diffusion.

\smallskip


%%%%%%%%%%%%%%%%%%%%%%%%%%%%%%%%%%%%%%%%%

\section{Introduction}
\label{introduction}

The problem of a company that aims at determining the extraction rule of an exhaustible commodity, while maximizing net profits, has been widely studied in the literature. To the best of our knowledge, the first model on this topic is the seminal paper \cite{Hotelling}, in which a deterministic model of optimal extraction has been proposed. Since then, many authors have generalized the setting of \cite{Hotelling} by allowing for stochastic commodity prices and for different specifications of the admissible extraction rules (see, e.g., \cite{almansour}, \cite{brekke}, \cite{feliz}, \cite{ferrari}, \cite{Pemy}, \cite{Pindyck1} and \cite{Pindyck2} among a huge literature in Economics and applied Mathematics).

In this paper, we consider an optimal extraction problem for an infinitely-lived profit maximizing company. The company extracts an exhaustible commodity from a reservoir with a finite capacity incurring constant proportional costs, and then immediately sells the commodity in the spot market. The admissible extraction rules must not be rates, also lump sum extractions are allowed. Moreover, we assume that the company is a large player in the market, and therefore its extraction strategies affect the market price of the commodity. This happens in such a way that whenever the company extracts the commodity and sells it in the market, the commodity's price is instantaneously decreased proportionally to the extracted amount.  

Our mathematical formulation of the previous problem leads to a \emph{two-dimensional degene-rate finite-fuel singular stochastic control problem} (see \cite{Bridge}, \cite{K85}, \cite{karatzas2} and \cite{Karatzasetal00} as early contributions, and \cite{Becherer} and \cite{guo} for recent applications to optimal liquidation problems). The underlying state variable is a two-dimensional process $(X,Y)$ whose components are the commodity's price and the level of the reservoir (i.e.\ the amount of commodity still available). The price process is a linearly controlled It\^o-diffusion, while the dynamics of the level of the reservoir are purely controlled and do not have any diffusive component. In particular, we assume that, in absence of any interventions, the commodity's price evolves either as a drifted Brownian motion or as an Ornstein-Uhlenbeck process, and we solve explicitly the optimal extraction problem by following a guess-and-verify approach. This relies on the construction of a classical solution to the associated Hamilton-Jacobi-Bellman (HJB) equation, which, in our problem, takes the form of a variational inequality with state-dependent gradient constraint. To the best of our knowledge, this is the first paper that provides the explicit solution to an optimal extraction problem under uncertainty for a price-maker company facing a diffusive commodity's spot price with additive and mean-reverting dynamics.

In the simpler case of a drifted Brownian dynamics for the commodity's price, we find that the optimal extraction rule prescribes at any time to extract just the minimal amount needed to keep the commodity's price below an endogenously determined constant critical level $x^{\star}$, the so-called free boundary. A lump sum extraction (and therefore a jump in the optimal control) may be observed only at initial time if the initial commodity's price exceeds the level $x^{\star}$. In such a case, depending on the initial level of the reservoir, it might be optimal either to deplete the reservoir or to extract a block of commodity so that the price is reduced to the desired level $x^{\star}$. 

If the commodity's price has additionally a mean-reverting behavior and evolves as an Ornstein-Uhlenbeck process, the analysis is much more involved and technical than in the Brownian case. This is due to the unhandy and not explicit form of the fundamental solutions to the second-order ordinary differential equation involving the infinitesimal generator of the Ornstein-Uhlenbeck process. The properties of the increasing fundamental solution are indeed needed when constructing an explicit solution to the HJB equation. The optimal extraction rule is triggered by a critical price level that - differently to the Brownian case - is not anymore constant, but it is depending on the current level of the reservoir $y$. This critical price level - that we call $F^{-1}(y)$ in Section \ref{sec:OU} - is the inverse of a positive, strictly decreasing, $C^{\infty}$-function $F$ that we determine explicitly.
It is optimal to extract in such a way that the joint process $(X,Y)$ is kept within the region $\{(x,y):\ x \leq F^{-1}(y)\}$, and a suitable lump sum extraction should be made only if the initial data lie outside the previous region. The free boundary $F$ has an asymptote at a point $x_{\infty}$ and it is zero at the point $x_0$. These two points have a clear interpretation, as they correspond to the critical price levels triggering the optimal extraction rule in a model with infinite fuel and with no market impact, respectively.

In both the Brownian and the Ornstein-Uhlenbeck case, the optimal extraction rule is mathematically given through the solution to a Skorokhod reflection problem with oblique reflection at the free boundary in the direction $(-\alpha,-1)$. Here $\alpha>0$ is the marginal market impact of the company's actions on the commodity's price. Indeed, if the company extracts an amount, say $d\xi_t$, at time $t$, then the price is linearly reduced by $\alpha d\xi_t$ and the level of the reservoir by $d\xi_t$. Moreover, we prove that the value function is a classical $C^{2,1}$-solution to the associated HJB equation.

When the price follows an Ornstein-Uhlenbeck dynamics, our proof of the optimality of the constructed candidate value function partly employs arguments developed in the study of an optimal liquidation problem tackled  in the recent \cite{Becherer}, which shares mathematical similarities with our problem. Indeed, in the case of a ``small'' marginal cost of extraction, due to the unhandy and implicit form of the increasing eigenfunction of the infinitesimal generator of the Ornstein-Uhlenbeck process, we have not been able to prove via direct means an inequality that the candidate value function needed to satisfy in order to solve the HJB equation. For this reason, in such a case, we adopted ideas from \cite{Becherer} where an interesting reformulation of the original singular control problem as a calculus of variations approach has been developed. However, it is also worth noticing that when the marginal cost of extraction is ``large enough'', the approach of \cite{Becherer} is not directly applicable since a fundamental assumption in \cite{Becherer} (cf.\ Assumption 2.2-(C5) therein) is not satisfied. Instead, a direct study of the variational inequality leads to the desired result. This fact suggests that a combined use of the calculus of variations method and of the standard guess-and-verify approach could be successful in intricate problems where neither of the two methods leads to prove optimality of a candidate value function for any choice of the model's parameters. We refer to the proof of Proposition \ref{HJB} and to Remark \ref{rem:verificoBech} for details.

As a byproduct of our results, we find that the directional derivative (in the direction $(-\alpha,-1)$) of the optimal extraction problem's value function coincides with the value function of an optimal stopping problem (see Section \ref{sec:relatedOS} and Remark \ref{rem:OS} below). This fact, which is consistent with the findings of \cite{K85} and \cite{karatzas2}, also allows us to explain quantitatively why, in the case of a drifted Brownian dynamics for the commodity's price, the level $x^{\star}$ triggering the optimal extraction rule is independent of the current level of the reservoir $y$. Indeed, in such a case, the value function of the optimal stopping problem is independent of $y$ and, therefore, so is also its free boundary $x^{\star}$.

Thanks to the explicit nature of our results, we can provide in Section \ref{sec:compstatics} a detailed comparative statics analysis. We obtain theoretical results on the dependency of the value function and of the critical price levels $x^{\star}$, $x_{\infty}$, and $x_0$ with respect to some of the model's parameters. In the case of an Ornstein-Uhlenbeck commodity's price, numerical results are also derived to show the dependency of the free boundary curve $F$ with respect to the volatility, the mean reversion level, and the mean-reversion speed.

The rest of the paper is organized as follows. In Section \ref{sec:setting} we introduce the setting and formulate the problem. In Section \ref{sec:PreResVerTheorem} we provide preliminary results and a Verification Theorem. The explicit solution to the optimal extraction problem is then constructed in Sections \ref{sec:ABM} and \ref{sec:OU} when the commodity's price is a drifted Brownian motion and an Ornstein-Uhlenbeck process, respectively. A connection to an optimal stopping problem is derived in Section \ref{sec:relatedOS}. A sensitivity analysis is presented in Section \ref{sec:compstatics}. The appendices contain the proofs of some results needed in Sections \ref{sec:OU} and \ref{comp:OU}, and an auxiliary lemma.

%%%%%%%%%%%%%%%%%%%%%%%%%%%%%%%%%%%%%%%%%%%%%%%%%%%%%%%%%%%%%%%%%%%%%%%%%%%%%%%%%%%%%%%%%%%%%%%%%%%%%%%%%%%%%%%%%%%%%%%%%%%%%%%%%%%%%%%%%%%%%%%%%%%%%%%%

\section{Setting and Problem Formulation}
\label{sec:setting}
Let $(\Omega, \mathcal{F}, \mathbb{F}:=(\mathcal{F}_t)_{t\geq 0}, \mathbb{P})$ be a filtered probability space, with filtration $\mathbb{F}$ generated by a standard one-dimensional Brownian motion $(W_t)_{t\geq 0}$, and as usual augmented by $\mathbb{P}$-null sets.

We consider a company extracting a commodity from a reservoir with a finite capacity $y \geq 0$, and selling it instantaneously in the spot market. We assume that, in absence of any interventions of the company, the (fundamental) commodity's price $(X^x_t)_{t\geq 0}$ evolves stochastically according to the dynamics 
\begin{align}
\label{Xuncontrolled}
dX^x_t = \big(a-bX^x_t\big)dt + \sigma dW_t,\qquad X^x_0 = x \in\mathbb{R},
\end{align} 
for some constants $a\in\mathbb{R}$, $b\geq 0$ and $\sigma>0$. In the following, we identify the fundamental price when $b=0$ with a drifted Brownian motion with drift $a$. On the other hand, when $b>0$ the price is of Ornstein-Uhlenbeck type, thus having a mean-reverting behavior typically observed in the commodity market (see, e.g., Chapter 2 of \cite{Lutz}). In this latter case, the parameter $\frac{a}{b}$ represents the mean-reversion level, and $b$ is the mean-reversion speed. In our model we do not restrict our attention to positive fundamental prices, since certain commodities have been traded also at negative prices. For example, that happened in Alberta (Canada) in October 2017 and May 2018 where the producers of natural gas faced the tradeoff between paying customers to take gas, or shutting down the wells\footnote{See, e.g., the article on the \href{http://business.financialpost.com/commodities/canadian-natural-gas-prices-enter-negative-territory-amid-pipeline-outages}{Financial Post} or the news on the website of the \href{https://www.eia.gov/naturalgas/weekly/archivenew_ngwu/2018/05_10/}{U.S.\ Energy Information Administration}}. 

The reserve level can be decreased at a constant proportional cost $c> 0$. The extraction does not need to be performed at a rate, and we identify the cumulative amount of commodity that has been extracted up to time $t\geq 0$, $\xi_t$, as the company's control variable. It is an $\mathbb{F}$-adapted, nonnegative, and increasing c\`{a}dl\`{a}g (right-continuous with left-limits) process $(\xi_t)_{t\geq 0}$ such that $\xi_t\leq y$ a.s.\ for all $t\geq 0$ and $\xi_{0-}=0$ a.s. The constraint $\xi_t\leq y$ for all $t\geq 0$ has the clear interpretation that at any time it cannot be extracted more than the initial amount of commodity available in the reservoir. For any given $y\geq0$, the set of \emph{admissible extraction strategies} is therefore defined as 
\begin{align*}
\mathcal{A}(y):=\{\xi:\Omega\times[0,\infty)\mapsto[0,\infty)&\text{ : $(\xi_t)_{t\geq 0}$ is $\mathbb{F}$-adapted, }t\mapsto\xi_t \text{ is increasing, c\`{a}dl\`{a}g,}\\
&\text{ with }\xi_{0-}=0\text{ and } \xi_t\leq y\text{ a.s.}\}.
\end{align*} 
Clearly, $\mathcal{A}(0)=\{\xi\equiv 0\}$.

The level of the reservoir at time $t$, $Y_t$, then evolves as 
$$dY^{y,\xi}_t=-d\xi_t,\qquad Y^{y,\xi}_{0-}=y\geq 0,$$ 
where we have written $Y^{y,\xi}$ in order to stress the dependency of the reservoir's level on the initial amount of commodity $y$ and on the extraction strategy $\xi$.

While extracting, the company affects the market price of the commodity. In particular, when following an extraction strategy $\xi\in\mathcal{A}(y)$, the market price at time $t$, $X_t$, is instantaneously reduced by $\alpha d\xi_t$, for some $\alpha>0$, and the spot price thus evolves as 
\begin{align}
\label{affectedX}
dX^{x,\xi}_t = \big(a-bX^{x,\xi}_t\big)dt + \sigma dW_t-\alpha d\xi_t,\qquad X^{x,\xi}_{0-} = x\in\mathbb{R}.
\end{align}
We notice that for any $\xi\in\mathcal{A}(y)$ there exists a unique strong solution to \eqref{affectedX} by Theorem 6 in Chapter V of \cite{protter}, and we denote it by $X^{x,\xi}$ in order to keep track of its initial value $x\in\mathbb{R}$, and of the adopted extraction strategy $\xi\in\mathcal{A}(y)$.

\begin{remark}
Notice that when $b=0$, the impact of the company's extraction on the price is permanent. On the other hand, it is transient (or temporary) in the mean-reverting case $b>0$ because, in the absence of any interventions from the company, the impact decreases since $X$ reverts back to its mean-reversion level.
\end{remark}

The company aims at maximizing the total expected profits, net of the total expected costs of extraction. That is, for any initial price $x \in \mathbb{R}$ and any initial value of the reserve $y\geq0$, the company aims at determining $\xi^{\star} \in \mathcal{A}(y)$ that attains
\begin{align}
\label{ValueFnc} 
V(x,y):=\mathcal{J}(x,y,\xi^{\star})=\sup_{\xi\in\mathcal{A}(y)}\mathcal{J}(x,y,\xi),
\end{align}
where 
\begin{align}\label{PC}
J(x,y,\xi):=\mathbb{E}\bigg[\int_{0}^{\infty}e^{-\rho t}(X_t^{x,\xi}-c) d\xi_t^c+\sum_{t\geq 0:\Delta\xi_t\neq 0}e^{-\rho t}\big[(X_{t-}^{x,\xi}-c)\Delta\xi_t-\frac{1}{2}\alpha(\Delta\xi_t)^2\big]\bigg],
\end{align}
for any $\xi\in\mathcal{A}(y)$, and for a given discount factor $\rho>0$. Here, and also in the following, $\Delta\xi_t:=\xi_t-\xi_{t-}$, $t\geq 0$, and $\xi^c$ denotes the continuous part of $\xi\in\mathcal{A}(y)$.

\begin{remark}
In \eqref{PC} the integral term in the expectation is intended as a standard Lebesgue-Stieltjes integral with respect to the continuous part $\xi^c$ of $\xi$. The sum takes instead care of the lump sum extractions, and its form might be informally justified by interpreting any lump sum extraction of size $\Delta\xi_t$ at a given time $t$ as a sequence of infinitely many infinitesimal extractions made at the same time $t$. In this way, setting $\epsilon_t:=\frac{\Delta\xi_t}{N}$, the net profit accrued at time $t$ by extracting a large amount $\Delta\xi_t$ of the commodity is 
$$\sum_{j=0}^{N-1}e^{-\rho t}\big(X_{t-}^{x,\xi}-c-j\alpha\epsilon_t\big)\epsilon_t\overset{N\rightarrow\infty}{\longrightarrow}\int_{0}^{\Delta\xi_t}e^{-\rho t} \big(X^{x,\xi}_{t-}-c-\alpha u\big)du=e^{-\rho t}\Big[\big(X_{t-}^{x,\xi}-c\big)\Delta\xi_t-\frac{1}{2}\alpha(\Delta\xi_t)^2\Big].$$ 
This heuristic argument - also discussed at pp.\ 329--330 of \cite{Alvarez2000} in the context of one-dimensional monotone follower problems - can be rigorously justified, and technical details on the convergence can be found in the recent \cite{Becherer2}. We also refer to \cite{Jacketal} and \cite{Zhu} as other papers on singular stochastic control problems employing such a definition for the integral with respect to the control process.
\end{remark}

%%%%%%%%%%%%%%%%%%%%%%%%%%%%%%%%%%%%%%%%%%%%%%%%%%%%%%%%%%%%%%%%%%%%%%%%%%%%%%%%%%%%%%%%%%%%%%%%%%%%%%%%%%%%%%%%%%%%%%%%%%%%%%%%%%%%%%%%%%%%%%%%%%%%%%%%%%%%%%%%%%%%%%%%%%%%%%%%%%%%%%%%%%%%%%%%%%%%%%%%%%%%%%%%%%

\section{Preliminary Results and a Verification Theorem}
\label{sec:PreResVerTheorem}

In this section we derive the HJB equation associated to $V$ and we provide a verification theorem. We start by proving the following preliminary properties of the value function $V$.
\begin{proposition}
\label{GrowthV}
	There exists a constant $K>0$ such that for all $(x,y)\in\mathbb{R}\times [0,\infty)$ one has
	\begin{align}
	\label{eq13}
	0\leq V(x,y)\leq Ky(1+y)\big(1+|x|\big).
	\end{align}
In particular, $V(x,0)=0$. Moreover, $V$ is increasing with respect to $x$ and $y$.	
\end{proposition}

\begin{proof}
	The proof is organized in two steps. We first prove that \eqref{eq13} holds true, and then we show the monotonicity properties of $V$.\vspace{0.25 cm}
	
	\emph{Step 1.} The nonnegativity of $V$ follows by taking the admissible (no-)extraction rule $\xi\equiv 0$ such that $\mathcal{J}(x,y,0)=0$ for all $(x,y)\in\mathbb{R}\times[0,\infty)$. The fact that $V(x,0)=0$ clearly follows by noticing that $\mathcal{A}(0)=\{\xi\equiv 0\}$ and $\mathcal{J}(x,y,0)=0$.
	
	To determine the upper bound in \eqref{eq13}, let $(x,y)\in\mathbb{R}\times(0,\infty)$ be given and fixed, and for any $\xi\in\mathcal{A}(y)$ we have 
	\begin{align}
	\label{eq11}
	\begin{split}
	&\bigg|\mathbb{E}\bigg[\int_{0}^{\infty}e^{-\rho t}\big(X^{x,\xi}_t-c\big)d\xi_t^c+\sum_{t\geq0:\Delta\xi_t\neq0}e^{-\rho t}\big[(X^{x,\xi}_{t-}-c)\Delta\xi_t-\frac{\alpha}{2}(\Delta\xi_t)^2\big]\bigg]\bigg|\\
	\leq&\,\mathbb{E}\bigg[\int_{0}^{\infty}e^{-\rho t}|X^{x,\xi}_t|d\xi_t^c\bigg]+cy+\mathbb{E}\bigg[\sum_{t\geq0:\Delta\xi_t\neq0}e^{-\rho t}\big[|X^{x,\xi}_{t-}|\Delta\xi_t+\frac{\alpha}{2}(\Delta\xi_t)^2\big]\bigg],
	\end{split}
	\end{align}
	where we have used that $c\int_{0}^{\infty}e^{-\rho t} d\xi_t = c \int_{0}^{\infty}\rho e^{-\rho t} \xi_t dt \leq cy$ to obtain the term $cy$ in right-hand side above.
	
	We now aim at estimating the two expectations appearing in right-hand side of $\eqref{eq11}$. To accomplish that, denote by $X^{x,0}$ the solution to \eqref{affectedX} associated to $\xi\equiv 0$ (i.e.\ the solution to \eqref{Xuncontrolled}). Then, if $b=0$ one easily finds $X^{x,\xi}_t = X^{x,0}_t - \alpha \xi_t \geq -|X^{x,0}_t| - \alpha y$ a.s., since $\xi_t\leq y$ a.s. If $b>0$, because $X^{x,\xi}_t\leq X_t^{x,0}$ a.s.\ for all $t\geq 0$ and $\xi_t\leq y$ a.s., one has
	\begin{align*}
	X^{x,\xi}_t&=x+\int_0^t\big(a-bX^{x,\xi}_s\big)ds +\sigma W_t-\alpha\xi_t\geq x+\int_0^t\big(a-bX^{x,0}_s\big)ds +\sigma W_t-\alpha y\\
	&=X^{x,0}_t-\alpha y\geq -|X_t^{x,0}|-\alpha y.
	\end{align*} 
	Moreover, one clearly has $X^{x,\xi}_t\leq X_t^{x,0}\leq |X_t^{x,0}|+\alpha y$ for $b\geq0$. Hence, in any case,
	\begin{align}
	\label{eq12}
	|X^{x,\xi}_t|\leq |X_t^{x,0}|+\alpha y.
	\end{align} 
	
	By an application of It\^o's formula we find for $b=0$ that
	$$|e^{-\rho t}X^{x,0}_t|\leq |x|+\rho\int_0^te^{-\rho u}|X^{x,0}_u|du+|a|\int_0^te^{-\rho u}du+\bigg|\int_0^te^{-\rho u}\sigma dW_u\bigg|,$$
	and for $b>0$ that
	$$|e^{-\rho t}X^{x,0}_t|\leq |x|+\rho\int_0^te^{-\rho u}|X^{x,0}_u|du+\int_0^te^{-\rho u}(|a|+b|X^{x,0}_u|)du+\bigg|\int_0^te^{-\rho u}\sigma dW_u\bigg|.$$
	The previous two equations imply that, in both cases $b=0$ and $b>0$, there exists $C_1>0$ such that
	\begin{align}
	\label{eq40}
	\mathbb{E}\bigg[\sup_{t\geq 0}e^{-\rho t}|X_t^{x,0}|\bigg]
	\leq|x|+C_1\bigg(1+\int_{0}^{\infty}e^{-\rho u}\mathbb{E}\big[|X^{x,0}_u|\big] du\bigg)+\sigma\mathbb{E}\bigg[\sup_{t\geq 0}\bigg|\int_{0}^{t}e^{-\rho u}dW_u\bigg|\bigg].
	\end{align}
	Then, the Burkholder-Davis-Gundy's inequality (see, e.g., Theorem 3.28 in Chapter 3 of \cite{karatzas}) yields 
	 \begin{align}
	 \label{eq41}
	 \mathbb{E}\bigg[\sup_{t\geq 0}e^{-\rho t}|X_t^{x,0}|\bigg]\leq |x|+C_1\bigg(1+\int_{0}^{\infty}e^{-\rho u}\mathbb{E}\big[|X^{x,0}_u|\big] du\bigg)+C_2\mathbb{E}\bigg[\bigg(\int_{0}^{\infty}e^{-2\rho u}du\bigg)^{\frac{1}{2}}\bigg].
	 \end{align}
	 for a constant $C_2>0$, and therefore
	\begin{align}
	\label{supfinite}
	\mathbb{E}\bigg[\sup_{t\geq 0}e^{-\rho t}|X_t^{x,0}|\bigg]\leq C_4\big(1+|x|\big),
	\end{align}
	for some constant $C_4>0$, since it follows from standard considerations that there exists $C_3>0$ such that $\int_0^{\infty}e^{-\rho u}\mathbb{E}\big[|X_u^{x,0}|\big] du \leq C_3(1+|x|)$.
	
	Now, exploiting \eqref{eq12} and \eqref{supfinite}, in both cases $b=0$ and $b>0$ we have the following:
	\begin{itemize}
		\item [(i)] For a suitable constant $K_0>0$ (independent of $x$ and $y$)
		\begin{align}\label{eq17}
		\begin{split}
		&\mathbb{E}\bigg[\int_{0}^{\infty}e^{-\rho t}|X^{x,\xi}_t|d\xi_t^c\bigg] \leq \mathbb{E}\bigg[\int_{0}^{\infty}e^{-\rho t}|X^{x,0}_t|d\xi_t^c\bigg] + \alpha y \mathbb{E}\bigg[\int_{0}^{\infty}\rho e^{-\rho t}\xi_t^c dt\bigg]  \\
		& \leq y\mathbb{E}\bigg[\sup_{t\geq 0}e^{-\rho t}|X_t^{x,0}|\bigg] + \alpha y^2 \leq C_4 y\big(1+|x|\big) + \alpha y^2 \leq K_0y(1+y)\big(1+|x|\big).  
		\end{split}
		\end{align}
		Here we have used: \eqref{eq12} and an integration by parts for the first inequality; the fact that $\xi_t^c  \leq y$ a.s.\ for the second one; equation \eqref{supfinite} to have the penultimate step.
		
		\item [(ii)] Employing again \eqref{eq12}, the fact that $\sum_{t\geq0: \Delta\xi_t\neq 0}\Delta \xi_t \leq y$, and \eqref{supfinite}, we also have
		\begin{align}\label{eq18}
		\begin{split}
		&\mathbb{E}\bigg[\sum_{t\geq0: \Delta\xi_t\neq 0}e^{-\rho t}\big[|X^{x,\xi}_{t-}|\Delta\xi_t+\frac{\alpha}{2}(\Delta\xi_t)^2\big]\bigg]\leq \frac{3}{2}\alpha y^2+\mathbb{E}\bigg[\sum_{t\geq 0: \Delta\xi_t\neq 0}e^{-\rho t}|X^{x,0}_{t}|\Delta\xi_t\bigg]\\
		&\,\leq\frac{3}{2}\alpha y^2+y\mathbb{E}\bigg[\sup_{t\geq 0}\big(e^{-\rho t}|X^{x,0}_t|\big)\bigg]\leq\frac{3}{2}\alpha y^2+C_4 y\big(1+|x|\big)\leq K_1y(1+y)\big(1+|x|\big),
		\end{split}
		\end{align}
		for some $K_1>0$.
	\end{itemize}
	Thus, using (i) and (ii) in \eqref{eq11}, we conclude that there exists a constant $K>0$ such that $|\mathcal{J}(x,y,\xi)|\leq Ky(1+y)\big(1+|x|\big)$ for any $\xi\in\mathcal{A}(y)$, and therefore \eqref{eq13} holds.\vspace{0.25cm}
	
	\emph{Step 2.} To prove that $x\mapsto V(x,y)$ is increasing for any $y\geq0$, let $x_2 \geq x_1$, and observe that one clearly has $X^{x_2,\xi}_t\geq X^{x_1,\xi}_t$ a.s.\ for any $t\geq 0$ and $\xi\in\mathcal{A}(y)$. Therefore $\mathcal{J}(x_2,y,\xi)\geq\mathcal{J}(x_1,y,\xi)$ which implies $V(x_2,y)\geq V(x_1,y)$. Finally, letting $y_2 \geq y_1$, we have $\mathcal{A}(y_2)\supseteq\mathcal{A}(y_1)$, and thus $V(x,y_2)\geq V(x,y_1)$ for any $x \in \mathbb{R}$.
\end{proof}

We now move on by deriving the dynamic programming equation that we expect that $V$ should satisfy. In the rest of this paper, we will often denote by $f_x,f_{xx},f_y,f_{xy}$ etc.\ the partial derivatives with respect to its arguments $x$ and $y$ of a given smooth function $f$ of several variables. Moreover, we will denote (unless otherwise stated) by $f'$, $f''$ etc.\ the derivatives with respect to its argument of a smooth function $f$ of a single variable.

At initial time the company is faced with two possible actions: extract or wait. On the one hand, suppose that at time zero the company does not extract for a short time period $\Delta t$, and then it continues by following the optimal extraction rule (if one exists). Since this action is not necessarily optimal, it is associated to the inequality $$V(x,y)\geq\mathbb{E}\bigg[e^{-\rho \Delta t}V(X^{x}_{\Delta t-},y)\bigg],\quad(x,y)\in\mathbb{R}\times(0,\infty).$$
Then supposing $V$ is $C^{2,1}(\mathbb{R}\times [0,\infty))$, we can apply It\^o's formula, divide by $\Delta t$, invoke the mean value theorem, let $\Delta t\rightarrow 0$, and obtain
$$\mathcal{L}V(x,y)-\rho V(x,y)\leq 0,\quad(x,y)\in\mathbb{R}\times(0,\infty).$$ Here $\mathcal{L}$ is given by the second order differential operator 
\begin{align}
\label{InfinitesOp}
\mathcal{L} :=\frac{1}{2}\sigma^2 \frac{\partial^2}{\partial x^2} + 
\begin{cases}
\displaystyle (a-bx) \frac{\partial}{\partial x} ,\quad & \text{if}\,\, b>0,\\ 
\\
\displaystyle a \frac{\partial}{\partial x},\quad & \text{if}\,\, b=0.
\end{cases}
\end{align}

On the other hand, suppose that the company immediately extracts an amount $\varepsilon>0$ of the commodity, sells it in the market, and then follows the optimal extraction rule (provided that one exists). With reference to \eqref{PC}, this action is associated to the inequality 
$$V(x,y)\geq V(x-\alpha\varepsilon,y-\varepsilon)+(x-c)\varepsilon-\frac{1}{2}\alpha\varepsilon^2,$$ 
which, adding and substracting $V(x-\alpha\varepsilon,y)$, dividing by $\varepsilon$, and letting $\varepsilon\rightarrow 0$, yields $$0\geq-\alpha V_x(x,y)-V_y(x,y)+x-c.$$ 

Since only one of those two actions can be optimal, and given the Markovian nature of our setting, the previous inequalities suggest that $V$ should identify with an appropriate solution $w$ to the Hamilton-Jacobi-Bellman (HJB) equation
\begin{align}\label{HJB1}
\max\Big\{\mathcal{L}w(x,y)-\rho w(x,y),-\alpha w_x(x,y)-w_y(x,y)+x-c\Big\}=0,\quad(x,y)\in\mathbb{R}\times(0,\infty),
\end{align} with boundary condition $w(x,0)=0$ (cf.\ Proposition \ref{GrowthV}), and satisfying the growth condition in \eqref{eq13}. Equation \eqref{HJB1} takes the form of a variational inequality with state-dependent gradient constraint.

With reference to \eqref{HJB1} we introduce the \emph{waiting region} 
\begin{equation}
\label{def:waiting}
\mathbb{W}:=\{(x,y)\in\mathbb{R}\times(0,\infty):\mathcal{L}w(x,y)-\rho w(x,y)=0,\, -\alpha w_x(x,y)-w_y(x,y)+x-c<0\},
\end{equation}
in which we expect that it is not optimal to extract the commodity, and the \emph{selling region} 
\begin{equation}
\label{def:selling}
\mathbb{S}:=\{(x,y)\in\mathbb{R}\times(0,\infty):\mathcal{L}w(x,y)-\rho w(x,y) \leq 0,\, -\alpha w_x(x,y)-w_y(x,y)+x-c=0\},
\end{equation}
where it should be profitable to extract and sell the commodity. In the following, we will denote by $\overline{\mathbb{W}}$ the topological closure of $\mathbb{W}$.

The next theorem shows that a suitable solution to HJB equation \eqref{HJB1} identifies with the value function, whenever there exists an admissible extraction rule that keeps (with minimal effort) the state process $(X,Y)$ inside $\overline{\mathbb{W}}$.

\begin{theorem}[Verification Theorem]
\label{VerificationTheorem}
Suppose there exists a function $w:\mathbb{R}\times[0,\infty)\mapsto\mathbb{R}$ such that $w\in C^{2,1}(\mathbb{R}\times[0,\infty))$, solves HJB equation \eqref{HJB1} with boundary condition $w(x,0)=0$, is increasing in $y$, and satisfies the growth condition 
	\begin{align}\label{eq15}
	0 \leq w(x,y)\leq Ky(1+y)(1+|x|), \quad (x,y) \in \mathbb{R}\times (0,\infty),
	\end{align} 
	for some constant $K>0$. Then $w\geq V$ on $\mathbb{R}\times[0,\infty)$.
	
	Moreover, suppose that for all initial values $(x,y)\in\mathbb{R}\times(0,\infty)$, there exists a process $\xi^\star\in\mathcal{A}(y)$ such that
	\begin{align}\label{eq2VT}
	(X_t^{x,\xi^\star},Y_t^{y,\xi^\star})\in\overline{\mathbb{W}},\quad\text{for all }t\geq 0,\, \text{$\mathbb{P}$-a.s.},\\\label{eq3VT}
	\xi_t^\star=\int_{[0,t]}\mathds{1}_{\{(X_s^{x,\xi^\star},Y_s^{y,\xi^\star})\in\mathbb{S}\}}d\xi^\star_s,\quad\text{for all }t\geq 0,\, \text{$\mathbb{P}$-a.s.}
	\end{align}
	Then we have $w=V$ on $\mathbb{R}\times[0,\infty)$ and $\xi^\star$ is optimal; that is, $\mathcal{J}(x,y,\xi^\star)=V(x,y)$ for all $(x,y) \in \mathbb{R}\times [0,\infty)$.
\end{theorem}
	
\begin{proof}The proof is organized in two steps. Since by assumption $w(x,0)=0=V(x,0)$, $x \in \mathbb{R}$, in the following argument we can assume that $y>0$.
\vspace{0.25cm}

	\emph{Step 1.} Let $(x,y)\in\mathbb{R}\times(0,\infty)$ be given and fixed. Here, we show that $V(x,y)\leq w(x,y)$. Let $\xi\in\mathcal{A}(y)$, and for $N\in\mathbb{N}$ set $\tau_{R,N}:=\inf\{s\geq 0:X^{x,\xi}_s\notin (-R,R)\}\wedge N.$ By It\^o-Tanaka-Meyer's formula, we find
			\begin{align}\label{eq14}
			\begin{split}
			&e^{-\rho \tau_{R,N}}w(X^{x,\xi}_{\tau_{R,N}},Y^{y,\xi}_{\tau_{R,N}})-w(x,y)\\
			=&\int_{0}^{\tau_{R,N}}e^{-\rho s}\Big(\mathcal{L}w(X^{x,\xi}_s,Y^{y,\xi}_s)-\rho w(X^{x,\xi}_s,Y^{y,\xi}_s)\Big)ds+\underbrace{\sigma\int_{0}^{\tau_{R,N}}e^{-\rho s}w_x(X^{x,\xi}_s,Y^{y,\xi}_s)dW_s}_{=:M_{\tau_{R,N}}}\\&
			+\sum_{0\leq s\leq \tau_{R,N}}e^{-\rho s}\big[w(X^{x,\xi}_s,Y^{y,\xi}_s)-w(X^{x,\xi}_{s-},Y^{y,\xi}_{s-})\big]\\
			&+\int_{0}^{\tau_{R,N}}e^{-\rho s}\big[-\alpha w_x(X^{x,\xi}_s,Y^{y,\xi}_s) - w_y(X^{x,\xi}_s,Y^{y,\xi}_s)\big]d\xi_s^c.
			\end{split}
			\end{align}
			Now, 
			\begin{align*}
			& w(X^{x,\xi}_s,Y^{y,\xi}_s)-w(X^{x,\xi}_{s-},Y^{y,\xi}_{s-})=w(X^{x,\xi}_{s-}-\alpha\Delta\xi_s,Y^{y,\xi}_{s-}-\Delta\xi_s)-w(X^{x,\xi}_{s-},Y^{y,\xi}_{s-})\\
			&=\int_{0}^{\Delta\xi_s}\frac{\partial w(X^\xi_{s-}-\alpha u,Y^{y,\xi}_{s-}-u)}{\partial u}du\\
			&=\int_{0}^{\Delta\xi_s}\Big[-\alpha w_x(X^{x,\xi}_{s-}-\alpha u,Y^{y,\xi}_{s-}-u)-w_y(X^{x,\xi}_{s-}-\alpha u,Y^{y,\xi}_{s-}-u)\Big]du,
			\end{align*}
			which used into \eqref{eq14} gives the equivalence
			\begin{align*}
			&\int_{0}^{\tau_{R,N}}e^{-\rho s}\big(X^{x,\xi}_s-c\big)d\xi_s^c+\sum_{0\leq s\leq \tau_{R,N}}e^{-\rho s}\int_{0}^{\Delta\xi_s}\big(X^{x,\xi}_{s-}-\alpha u - c\big)du - w(x,y)\\
			=& \hspace{0.25cm} - e^{-\rho \tau_{R,N}}w(X^{x,\xi}_{\tau_{R,N}},Y^{y,\xi}_{\tau_{R,N}})+\int_{0}^{\tau_{R,N}}e^{-\rho s}\Big(\mathcal{L}w(X^{x,\xi}_s,Y^{y,\xi}_s)-\rho w(X^{x,\xi}_s,Y^{y,\xi}_s)\Big)ds+M_{\tau_{R,N}}\\&
			+\sum_{0\leq s\leq \tau_{R,N}}e^{-\rho s}\int_{0}^{\Delta\xi_s}\Big[-\alpha w_x(X^{x,\xi}_{s-}-\alpha u,Y^{y,\xi}_{s-}-u)-w_y(X^{x,\xi}_{s-}-\alpha u,Y^{y,\xi}_{s-}-u) \\
			& \hspace{0.25cm} + (X^{x,\xi}_{s-}-\alpha u-c)\Big]du +\int_{0}^{\tau_{R,N}}e^{-\rho s}\Big[-\alpha w_x(X^{x,\xi}_s,Y^{y,\xi}_s)-w_y(X^{x,\xi}_s,Y^{y,\xi}_s)+X^{x,\xi}_s-c\Big]d\xi_s^c.
			\end{align*}
			 Since $w$ satisfies \eqref{HJB1} and $w\geq 0$, by taking expectations on both sides of the latter equation, and using that $\mathbb{E}[M_{\tau_{R,N}}]=0$, we have
			 \begin{align}
			 w(x,y) \geq \mathbb{E}\Big[\int_{0}^{\tau_{R,N}}e^{-\rho s}\big(X^{x,\xi}_s - c\big)d\xi_s^c +\sum_{0\leq s\leq \tau_{R,N}}e^{-\rho s}\int_{0}^{\Delta\xi_s}\big(X^{x,\xi}_{s-}-\alpha u-c\big)du\Big].
			 \end{align}
			
			 We now want to take limits as $N\uparrow \infty$ and $R\uparrow \infty$ on the right-hand side of the equation above. To this end notice that one has a.s.
			 \begin{align}
			\label{eq19}
			\begin{split}
			 &\Big|\int_{0}^{\tau_{R,N}}e^{-\rho s}\big(X^{x,\xi}_s-c\big)d\xi_s^c +\sum_{0\leq s\leq \tau_{R,N}}e^{-\rho s}\int_{0}^{\Delta\xi_s}
			\big(X^{x,\xi}_{s-}-\alpha u - c\big)du\Big| \\
			 \leq\,&\int_{0}^{\infty}e^{-\rho s}|X^{x,\xi}_s|d\xi_s^c+cy+\sum_{s\geq 0:\Delta\xi_s\neq 0}e^{-\rho s}\big(|X^{x,\xi}_{s-}|\Delta\xi_s+\frac{\alpha}{2}(\Delta\xi_s)^2\big),
			\end{split}
			 \end{align}
			 and the right-hand side of \eqref{eq19} is integrable by \eqref{eq17} and \eqref{eq18}. Hence, we can invoke the dominated convergence theorem in order to take limits as $R\uparrow\infty$ and then as $N\uparrow\infty$, so as to get
			\begin{align}\label{eqVT}
			\mathcal{J}(x,y,\xi)\leq w(x,y).
			\end{align}
			Since $\xi\in\mathcal{A}(y)$ is arbitrary, we have
			\begin{align}\label{eq58}
			V(x,y)\leq w(x,y),
			\end{align}
			which yields $V\leq w$ by arbitrariness of $(x,y)$ in $\mathbb{R}\times(0,\infty)$.
			\vspace{0.25cm}
			 
			\emph{Step 2.} Here, we prove that $V(x,y)\geq w(x,y)$ for any $(x,y)\in\mathbb{R}\times(0,\infty)$. Let $\xi^\star\in\mathcal{A}(y)$ satisfying \eqref{eq2VT} and \eqref{eq3VT}, and let $\tau^\star_{R,N}:=\inf\{t\geq 0: X^{x,\xi^\star}_{t}\notin(-R,R)\}\wedge N$, for $N\in\mathbb{N}$. Then, by employing the same arguments as in \emph{Step 1}, all the inequalities become equalities and we obtain 
			\begin{align*}
			&\mathbb{E}\bigg[\int_{0}^{\tau^\star_{R,N}}e^{-\rho s}\big(X^{x,\xi^\star}_s-c\big)d\xi_s^{\star,c}+\sum_{0\leq s\leq \tau^\star_{R,N}}e^{-\rho s}\int_{0}^{\Delta\xi^\star_s}\big(X^{x,\xi^\star}_{s-}-c-\alpha u\big)du\bigg]\\
			&+\mathbb{E}\Big[e^{-\rho \tau^\star_{R,N}}w(X^{x,\xi^\star}_{\tau^\star_{R,N}},Y^{\xi^\star}_{\tau^\star_{R,N}})\Big]= w(x,y),
			\end{align*}
			where $\xi^{\star,c}$ denotes the continuous part of $\xi^\star$. If now
			\begin{align}\label{eq53}
			\lim_{N\uparrow\infty}\lim_{R\uparrow\infty}\mathbb{E}\Big[e^{-\rho \tau^\star_{R,N}}w(X^{x,\xi^\star}_{\tau^\star_{R,N}},Y^{\xi^\star}_{\tau^\star_{R,N}})\Big]= 0,
			\end{align}
			then we can take limits as $R\uparrow\infty$ and $N\uparrow\infty$, and by \eqref{eq19} (with $\xi=\xi^\star$) together with \eqref{eq17} and \eqref{eq18} we find $\mathcal{J}(x,y,\xi^\star)= w(x,y)$. Since clearly $V(x,y)\geq\mathcal{J}(x,y,\xi^\star)$, then $V(x,y)\geq w(x,y)$ for all $(x,y)\in\mathbb{R}\times(0,\infty)$. Hence, using \eqref{eq58}, $V=w$ on $\mathbb{R}\times (0,\infty)$, and therefore on $\mathbb{R}\times [0,\infty)$ because $V(x,0)=0=w(x,0)$ for all $x\in\mathbb{R}$. 
			
			To complete the proof it thus only remains to prove \eqref{eq53}, and we accomplish that in the following. Since $y\mapsto w(x,y)$ is increasing by assumption, we have by \eqref{eq15} and \eqref{eq12} that
			 \begin{align*}
			 0 \leq e^{-\rho\tau^\star_{R,N}}w(X^{x,\xi^\star}_{\tau^\star_{R,N}},Y^{\xi^\star}_{\tau^\star_{R,N}})&\leq e^{-\rho\tau^\star_{R,N}}w(X^{x,\xi^\star}_{\tau^\star_{R,N}},y)\leq e^{-\rho\tau^\star_{R,N}}Ky(1+y)\big(1+|X^{x,\xi^\star}_{\tau^\star_{R,N}}|\big)\\
			 &\leq Ky(1+y)\big[(1+\alpha y)e^{-\rho\tau^\star_{R,N}}+e^{-\rho\tau^\star_{R,N}}|X^{x,0}_{\tau^\star_{R,N}}|\big]\\
			 &\leq Ky(1+y)\big[(1+\alpha y)e^{-\rho\tau^\star_{R,N}}+e^{-\frac{\rho}{2}\tau^\star_{R,N}}\sup\limits_{t\geq 0}e^{-\frac{\rho}{2} t}|X^{x,0}_t|\big].
			 \end{align*} Taking expectations and employing H\"older's inequality
			 \begin{align}
			\label{eq39}
			 \begin{split}
			 &0 \leq \mathbb{E}\big[e^{-\rho \tau^\star_{R,N}}w(X^{x,\xi^\star}_{\tau^\star_{R,N}},Y^{\xi^\star}_{\tau^\star_{R,N}})\big]\\
			 &\leq Ky(1+y)\Big[(1+\alpha y)\mathbb{E}\big[e^{-\rho \tau^\star_{R,N}}\big]+\mathbb{E}\Big[e^{-{\rho} \tau^\star_{R,N}}\Big]^{\frac{1}{2}}\mathbb{E}\Big[\sup\limits_{t\geq 0}e^{-\rho t}|X^{x,0}_t|^2\Big]^{\frac{1}{2}}\Big].
			 \end{split}
			 \end{align}
			 To take care of the third expectation on right hand side of \eqref{eq39}, observe that by It\^o's formula we have (in both cases $b=0$ and $b>0$)
			 \begin{align}
			 \begin{split}\label{eq59}
			 e^{-\rho t}(X^{x,0}_t)^2\leq x^2&+ \int_0^te^{-\rho u}\Big[\rho(X^{x,0}_u)^2+\sigma^2\Big]du\\&+\int_0^t2e^{-\rho u}|X^{x,0}_u|(|a|+b|X^{x,0}_u|)du+2\sigma\sup\limits_{t\geq 0}\bigg|\int_0^te^{-\rho u} X^{x,0}_u dW_u\bigg|.
			 \end{split}
			 \end{align}
			 Notice that $\int_0^{\infty}e^{-2\rho u}\mathbb{E}\big[|X_u^{x,0}|^2\big] du \leq C_1(1+|x|^2)$, for some constant $C_1>0$, and therefore an application of the Burkholder-Davis-Gundy's inequality (see, e.g., Theorem 3.28 in \cite{karatzas}) gives
			 \begin{align}\label{eq43}
			 \mathbb{E}\Big[\sup_{t\geq 0}\Big|\int_0^te^{-\rho u} X^{x,0}_u dW_u\Big|\Big]\leq C_2(1+|x|),
			 \end{align}
			 for a suitable $C_2>0$. Then taking expectations in \eqref{eq59}, employing \eqref{eq43}, we easily obtain that there exists a constant $C_3>0$ such that $$\mathbb{E}\big[\sup\limits_{t\geq 0}e^{-\rho t}|X^{x,0}_t|^2\big]\leq C_3(1+|x|^2).$$ 
			 Hence, when taking limits as $R\uparrow\infty$ and $N\uparrow\infty$ in \eqref{eq39}, the right-hand side of \eqref{eq39} converges to zero, thus proving \eqref{eq53} and completing the proof.		 
\end{proof}
		
\section{Constructing the Optimal Solution}	

We make the guess that the company extracts and sells the commodity only when the current price is sufficiently large. We therefore expect that for any $y>0$ there exists a critical price level $G(y)$ (to be endogenously determined) separating the \emph{waiting region} $\mathbb{W}$ and the \emph{selling region} $\mathbb{S}$ (cf.\ \eqref{def:waiting} and \eqref{def:selling}). In particular, we suppose that
\begin{align}
\mathbb{W}&=\{(x,y)\in\mathbb{R}\times(0,\infty)\,:\, y>0\text{ and }x<G(y)\}\cup(\mathbb{R}\times\{0\}),\\
\mathbb{S}&=\{(x,y)\in\mathbb{R}\times(0,\infty)\,:\, y>0\text{ and }x \geq G(y)\}.
\end{align}

According to such a guess, and with reference to \eqref{HJB1}, the candidate value function $w$ should satisfy 
\begin{align}\label{cond1}
\mathcal{L}w(x,y)-\rho w(x,y)=0, \qquad \mbox{for all}\,\,(x,y) \in \mathbb{W}.
\end{align}
It is well known that \eqref{cond1} admits two fundamental strictly positive solutions $\varphi(x)$ and $\psi(x)$, with the former one being strictly decreasing and the latter one being strictly increasing. Therefore, any solution to \eqref{cond1} can be written as $$w(x,y)=A(y)\psi(x)+B(y)\varphi(x),\quad(x,y)\in\mathbb{W},$$ for some functions $A(y)$ and $B(y)$ to be found. In both cases $b=0$ and $b>0$ (cf.\ \eqref{affectedX}), the function $\varphi$ increases exponentially to $+\infty$ as $x\downarrow-\infty$ (see, e.g., Appendix 1 in \cite{BorSal}). In light of the growth conditions of $V$ proved in Proposition \ref{GrowthV}, we therefore guess $B(y) = 0$ so that 
\begin{align}\label{eq45}
w(x,y)=A(y)\psi(x)
\end{align} for any $(x,y)\in\mathbb{W}$.

For all $(x,y)\in\mathbb{S}$, $w$ should instead satisfy
\begin{align}\label{cond2}
-\alpha w_x(x,y)-w_y(x,y)+x-c=0,
\end{align}
implying 
\begin{align}\label{cond3}
-\alpha w_{xx}(x,y)-w_{yx}(x,y)+1=0.
\end{align}
 
To find $G(y)$ and $A(y)$, $y>0$, we impose that $w \in C^{2,1}$, and therefore by \eqref{eq45}, \eqref{cond2}, and \eqref{cond3} we obtain for all $(x,y)\in\overline{\mathbb{W}}\cap\mathbb{S}$, i.e. $x=G(y)$, that
\begin{align}
\label{eq1}
-\alpha A(y)\psi'(x)-A'(y)\psi(x)+x-c&=0 \quad \text{at} \quad {x=G(y)},\\
\label{eq2}
-\alpha A(y)\psi''(x)-A'(y)\psi'(x)+1& =0 \quad \text{at} \quad {x=G(y)}.
\end{align}
From \eqref{eq1} and \eqref{eq2} one can easily derive that $A(y)$ and $G(y)$, $y>0$, satisfy
\begin{align}\label{cond4}
-\alpha A(y)\left(\psi'(x)^2-\psi(x)\psi''(x)\right)+(x-c)\psi'(x)-\psi(x)= 0 \quad \text{at} \quad {x=G(y)}.
\end{align}

In the following we continue our analysis by studying separately the cases $b=0$ and $b>0$, corresponding to a fundamental price of the commodity that is a drifted Brownian motion and an Ornstein-Uhlenbeck process, respectively. We will see that the form of the optimal extraction rule substantially differs among these two cases, and we will also provide a quantitative explanation of this by identifying an optimal stopping problem related to our optimal extraction problem (see Section \ref{sec:relatedOS} and Remark \ref{rem:OS} below).

%%%%%%%%%%%%%%%%%%%%%%%%%%%%%%%%%%%%%%%%%%%%%%%%%%%%%%%%%%%%%%%%%%%%%%%%%%%%%%%%%%%%%%%%%%%%%%%%%%%%%%%%%%%%%%%%%%%%%%%%%%%%%%%%%%%%%%%%%%%%%%%%%%%%%%%%%%%%%%%%%%%%%%%%%%%%%%%%%%%%%%%%%%%%%%%%%%%%%%%%%%%%%%%%%%

\subsection{$b=0$: The Case of a Drifted Brownian Motion Fundamental Price}
\label{sec:ABM}

We start with the simpler case $b=0$, and we therefore study the company's extraction problem \eqref{ValueFnc} when the fundamental commodity's price is a drifted Brownian motion. Dynamics \eqref{Xuncontrolled} with $b=0$ yield
$$dX^{x,\xi}_t = adt + \sigma dW_t-\alpha d\xi_t,\qquad X^{x,\xi}_{0-} = x\in\mathbb{R},$$
for any $\xi\in\mathcal{A}(y)$, and consequently \eqref{cond1} reads as 
\begin{align}
\label{eq48}
\frac{\sigma^2}{2}w_{xx}(x,y)+aw_x(x,y)-\rho w(x,y)=0,\quad (x,y)\in\mathbb{R}\times(0,\infty).
\end{align} 
The increasing fundamental solution $\psi$ to the latter equation is given by   
\begin{align}
\label{eq47}
\psi(x)=e^{nx} \qquad \text{ with } \qquad n:=-\frac{a}{\sigma^2}+\sqrt{\Big(\frac{a}{\sigma^2}\Big)^2+2\frac{\rho}{\sigma^2}}>0.
\end{align}
For future use, we notice that $n$ solves $B(n)=0$ with
\begin{align}\label{eqB}
B(u):=\frac{\sigma^2}{2}u^2+au-\rho,\quad u \in \mathbb{R}.
\end{align}

Upon observing that $\psi'(x)^2-\psi(x)\psi''(x)=0$ for all $x\in\mathbb{R}$, we see that any explicit dependency on $y$ disappears in \eqref{cond4}, and we therefore obtain that the critical price $G(y)$ identifies for any $y>0$ with the constant value
\begin{align}
\label{xstarABM}
x^\star=c+\frac{1}{n},
\end{align} 
which uniquely solves the equation $(x^\star-c)n-1=0$ (cf.\ \eqref{cond4} and \eqref{eq47}).

Moreover, by using either \eqref{eq1} or \eqref{eq2}, and by imposing $A(0)=0$ (since we must have $V(x,0)=0$ for all $x\in\mathbb{R}$; cf.\ Theorem \ref{VerificationTheorem}), the function $A$ in \eqref{eq45} is given by 
$$A(y):=\frac{1}{\alpha n^2}e^{-cn-1}\big(1-e^{-\alpha ny}\big),\quad y\geq 0.$$ 

In light of the previous findings, the candidate \emph{waiting region} $\mathbb{W}$ is given by
\begin{align*}
\mathbb{W}&=\{(x,y)\in\mathbb{R}\times (0,\infty) \,:\, y>0\text{ and }x<x^\star\}\cup(\mathbb{R}\times\{0\}),
\end{align*}
and we expect that the \emph{selling region} $\mathbb{S}$ is such that $\mathbb{S}=\mathbb{S}_1\cup\mathbb{S}_2,$ where
\begin{align*}
\mathbb{S}_1&:=\{(x,y)\in\mathbb{R}\times(0,\infty)\,:\, x\geq x^\star\text{ and }y\leq (x-x^\star)/{\alpha} \},\\
\mathbb{S}_2&:=\{(x,y)\in\mathbb{R}\times(0,\infty)\,:\, x\geq x^\star\text{ and }y> (x-x^\star)/{\alpha} \}.
\end{align*}
In $\mathbb{S}_1$, we believe that it is optimal to deplete immediately the reservoir. In $\mathbb{S}_2$ the company should make a lump sum extraction of size $(x-x^\star)/\alpha$, and then sell the commodity continuously and in such a way that the joint process $(X,Y)$ is kept inside $\overline{\mathbb{W}}$, until there is nothing left in the reservoir. These considerations suggest to introduce the candidate value function
\begin{align}
\label{CaValueABM}
w(x,y):=\begin{cases}
\frac{1}{\alpha n^2}e^{(x-c)n-1}(1-e^{-\alpha ny}),\quad &\text{if }(x,y)\in\mathbb{W},\\
\frac{1}{\alpha n^2}\left(1-e^{-\alpha n (y-\frac{x-x^\star}{\alpha})}\right)+(x-c)\big(\frac{x-x^\star}{\alpha}\big)-\frac{1}{2\alpha}(x-x^\star)^2\quad &\text{if }(x,y)\in\mathbb{S}_2,\\
(x-c)y-\frac{1}{2}\alpha y^2,\quad &\text{if }(x,y)\in\mathbb{S}_1.
\end{cases}
\end{align}
%where, to simplify exposition, we have set
%\begin{equation}
%\label{wxstar}
%w(x^\star,y-\frac{x-x^\star}{\alpha}):= \frac{1}{\alpha n^2}\left(1-e^{-\alpha n (y-\frac{x-x^\star}{\alpha})}\right).
%\end{equation}
Notice that the first term in the second line of \eqref{CaValueABM} is the continuation value starting from the new state $(x^\star,y-\frac{x-x^\star}{\alpha})$, and that $w$ above is continuous by construction. From now on, we will refer to the critical price level $x^{\star}$ as to the \emph{free boundary}. 

The next proposition shows that $w$ actually identifies with the value function $V$.
\begin{proposition}
	The function $w:\mathbb{R}\times[0,\infty)\mapsto[0,\infty)$ defined in \eqref{CaValueABM} is a $C^{2,1}(\mathbb{R}\times [0,\infty))$ solution to the HJB equation \eqref{HJB1} such that
	\begin{equation}
	\label{bounds-w}
	0 \leq w(x,y) \leq Ky(1+y)(1+|x|), \quad (x,y)\in \mathbb{R}\times [0,\infty),
	\end{equation}
	for a suitable constant $K>0$.
	
Moreover, it identifies with the value function $V$ from \eqref{ValueFnc}, and the admissible control 
\begin{align}
\label{ExtractABM}
\xi^\star_t:=y\wedge\sup_{0\leq s\leq t}\frac{1}{\alpha}\big[x-x^\star+ a s+\sigma W_s\big]^+,\quad t\geq 0, \qquad \xi^\star_{0-}=0,
\end{align}
with $x^{\star}$ as in \eqref{xstarABM}, is an optimal extraction strategy.
\end{proposition}

\begin{proof}
The proof is organized in steps. %For the following argument, it is important to bear in mind equation \eqref{wxstar} and that $\Big(x^\star,y-\frac{x-x^\star}{\alpha}\Big)\in\mathbb{W}$.

\vspace{0.25cm}

\emph{Step 1.} We start proving that $w\in C^{2,1}(\mathbb{R}\times[0,\infty))$. One can easily check that $w(x,0)=0$ for any $x\in\mathbb{R}$, and that $w$ is continuous on $\mathbb{R}\times[0,\infty)$ (recall also the comment after \eqref{CaValueABM}). For all $(x,y)\in\mathbb{W}$ we derive from \eqref{CaValueABM} 
\begin{align}\label{eq70}
w_x(x,y)=\frac{1}{\alpha n}e^{(x-c)n-1}(1-e^{-\alpha ny}),\quad w_{xx}(x,y)=\frac{1}{\alpha}e^{(x-c)n-1}(1-e^{-\alpha ny}), 
\end{align}
and
\begin{align}
\label{eq61}
w_y(x,y)=\frac{1}{n}e^{(x-c)n-1}e^{-\alpha ny}.
\end{align} 

%\begin{align}
%\label{eq61}
%w_y(x,y)=-{\alpha}w_x(x,y)+\frac{1}{n}e^{(x-c)n-1}.
%\end{align} 
Also, for all $(x,y)\in\mathbb{S}_2$ we find from \eqref{CaValueABM} by direct calculations that
\begin{align}\label{eq60}
w_x(x,y)=-\frac{1}{\alpha n}e^{-\alpha n (y-\frac{x-x^\star}{\alpha})}+\frac{x-c}{\alpha},\quad w_{xx}(x,y)=\frac{1}{\alpha}\left(1-e^{-\alpha n (y-\frac{x-x^\star}{\alpha})}\right),
\end{align}
and 
\begin{align}\label{eq73}
w_{y}(x,y)=\frac{1}{n}e^{-\alpha n (y-\frac{x-x^\star}{\alpha})}.
\end{align}
%Also, by recalling \eqref{xstarABM} and \eqref{eq61}, we find from \eqref{CaValueABM} that 
%\begin{align}
%\label{eq60}
%w_x(x,y)=w_x\Big(x^\star,y-\frac{x-x^\star}{\alpha}\Big)+\frac{1}{\alpha}(x-x^\star),\quad\text{for all $(x,y)\in\mathbb{S}_2$.}
%\end{align}
%Furthermore, we have from \eqref{eq61} by direct calculations that
%\begin{align}
%\label{eq51}
%w_{xy}(x,y)=-\alpha w_{xx}(x,y)+e^{(x-c)n-1},
%\end{align} 
%for all $(x,y)\in\mathbb{W}$ and therefore also for $\Big(x^\star,y-\frac{x-x^\star}{\alpha}\Big)\in\mathbb{W}$, and differentiating \eqref{eq60} with respect to $x$, we obtain with the help of \eqref{eq51} and \eqref{xstarABM} that
%\begin{align}\label{eq71}
%w_{xx}(x,y)=w_{xx}\Big(x^\star,y-\frac{x-x^\star}{\alpha}\Big),\quad(x,y)\in\mathbb{S}_2.
%\end{align}
%Moreover, \eqref{CaValueABM} also yields
%\begin{align}\label{eq73}
%w_{y}(x,y)=w_{y}\Big(x^\star,y-\frac{x-x^\star}{\alpha}\Big),\quad(x,y)\in\mathbb{S}_2.
%\end{align} 
Finally, for $(x,y)\in\mathbb{S}_1$ we have
\begin{align}\label{eq72}
w_x(x,y)=y,\quad w_{xx}(x,y)=0,\quad w_y(x,y)=x-c-\alpha y.
\end{align}
From the previous expressions it is now straightforward to check that $w\in C^{2,1}(\mathbb{R}\times [0,\infty))$ upon recalling $x^\star=c+\frac{1}{n}$ (cf.\ \eqref{xstarABM}).

\vspace{0.25cm}

\emph{Step 2.} Here we prove that $w$ solves HJB equation \eqref{HJB1}. By construction we have $-\alpha w_x(x,y)-w_y(x,y)+x-c=0$ for $(x,y)\in\mathbb{S}$, and $\cL w(x,y)-\rho w(x,y)=0$ for $(x,y)\in\mathbb{W}$. Hence it remains to prove that $-\alpha w_x(x,y)-w_y(x,y)+x-c\leq 0$ for $(x,y)\in\mathbb{W}$ and $\cL w(x,y)-\rho w(x,y)\leq 0$ for $(x,y)\in\mathbb{S}$. This is accomplished in the following.

On the one hand, letting $(x,y)\in\mathbb{W}$ we obtain from the first equation in \eqref{eq70} and \eqref{eq61} that
\begin{align*}
-\alpha w_x(x,y)-w_y(x,y)+x-c=-\frac{1}{n}e^{(x-c)n-1}+x-c\leq 0,
\end{align*}
where the last inequality is due to $e^{(x-c)n-1}\geq (x-c)n$, which derives from the well-known property of the exponential function $e^q\geq q+1$ for all $q\in\mathbb{R}$. 

On the other hand, for $(x,y)\in\mathbb{S}_1$ we find from the third line of \eqref{CaValueABM} and \eqref{eq72} that
\begin{align*}
\cL w(x,y)-\rho w(x,y)=ay-\rho(x-c)y+\frac{\alpha}{2}\rho y^2=:H_1(x,y).
\end{align*}
We now want to prove that $H_1(x,y)\leq 0$ for all $(x,y)\in\mathbb{S}_1$. Because $y\leq\frac{x-x^\star}{\alpha}$ with $x^\star=c+\frac{1}{n}$, we find
\begin{align*}
\frac{\partial H_1}{\partial y}(x,y)=a-\rho(x-c)+\alpha\rho y\leq a-\frac{\rho}{n}.
\end{align*}
In order to study the sign of $\frac{\partial H_1}{\partial y}$, we need to distinguish two cases.
If $a\leq 0$, then it follows immediately $\frac{\partial H_1}{\partial y}(x,y) \leq 0$. If $a>0$, then recall $B$ from \eqref{eqB} and notice that because $u\mapsto B(u)$ is increasing on $(-a/\sigma^2,\infty)\supset \mathbb{R}_+$, $B(n)=0$, and $B(\frac{\rho}{a})>0$, one has $\frac{\rho}{a}\geq n$.
Hence again $\frac{\partial H_1}{\partial y}(x,y)\leq 0$. Since now $\lim_{y\downarrow 0}H_1(x,y)=0$ for any $x\geq x^\star$, then we have just proved that $H_1(x,y)\leq 0$ for all $y\leq\frac{x-x^\star}{\alpha}$, and for any $x\geq x^\star$. Hence, $\cL w-\rho w\leq 0$ in $\mathbb{S}_1$.

Also, for $(x,y)\in\mathbb{S}_2$, we find 
\begin{align*}
\cL w(x,y)-\rho w(x,y)=\frac{a}{\alpha}(x-x^\star)-\rho(x-c)\big(\frac{x-x^\star}{\alpha}\big)+\frac{\rho}{2\alpha}(x-x^\star)^2=:H_2(x).
\end{align*}
To obtain the first equality in the equation above we have used the second line of \eqref{CaValueABM}, \eqref{eq60}, and that $n$ solves $B(n)=0$ with $B$ as in \eqref{eqB}. Notice that $H_2(x^\star)=0$ and $H_2^\prime(x)=\frac{1}{\alpha}\big(a-{\rho}(x-c)\big)$. If $a\leq 0$, we clearly have that $H_2'(x)\leq 0$, since $x\geq x^\star>c$. If $a>0$, then $H_2'(x)\leq 0$ if and only if $x\geq c+\frac{a}{\rho}$, but the latter inequality holds for any $x\geq x^\star$ since we have proved above that for $a>0$ we have $\frac{\rho}{a}\geq n$, and therefore, $x^\star=c+\frac{1}{n}\geq c+\frac{a}{\rho}$. Hence, in any case, $H_2'(x)\leq 0$ for all $x\geq x^\star$, and then $\cL w-\rho w\leq 0$ in $\mathbb{S}_2$. 

Combining all the previous findings we have that $w$ is a $C^{2,1}(\mathbb{R}\times [0,\infty))$ solution to the HJB equation \eqref{HJB1}.
\vspace{0.25cm}

\emph{Step 3.} Here we verify that $w$ satisfies all the requirements needed to apply Theorem \ref{VerificationTheorem}.

The fact that $y\mapsto w(x,y)$ is increasing in $\mathbb{W}$ and $\mathbb{S}_2$ easily follows from \eqref{eq61} and \eqref{eq73}, respectively. The monotonicity of $w(x,\cdot)$ in $\mathbb{S}_1$ is instead due to \eqref{eq72} and to the fact that $y \leq (x-x^{\star})/\alpha$ in $\mathbb{S}_1$ and $x^{\star}>c$.

In order to show the upper bound in \eqref{bounds-w}, notice that 
\begin{align}
\label{eq64}
w(x,y)\leq \frac{1}{\alpha n^2},\quad\text{for all $(x,y)\in\mathbb{W}$},
\end{align} 
since $x < x^{\star}$. Further, we find for all $(x,y)\in\mathbb{S}_2$ that 
\begin{align}\label{eq65}
\begin{split}
w(x,y)&=\frac{1}{\alpha n^2}\left(1-e^{-\alpha n (y-\frac{x-x^\star}{\alpha})}\right)+(x-c)\left(\frac{x-x^\star}{\alpha}\right)-\frac{1}{2\alpha}(x-x^\star)^2\\
&\leq\frac{1}{\alpha n^2}+(x-c)\left(\frac{x-x^\star}{\alpha}\right) \leq\frac{1}{\alpha n^2}+(x-c)y,
\end{split}
\end{align} 
where we have used that $y > (x-x^{\star})/\alpha$ for all $(x,y)\in\mathbb{S}_2$. Finally, for all $(x,y)\in\mathbb{S}_1$ it is clear that
\begin{align}\label{eq66}
w(x,y)=(x-c)y-\frac{1}{2}\alpha y^2\leq(x-c)y.
\end{align}
Hence, from \eqref{eq64}-\eqref{eq66} we see that $w$ satisfies the required growth condition.

We now show the nonnegativity of $w$. For all $(x,y)\in\mathbb{W}$ one clearly has $w(x,y)\geq 0$, and one also finds that for all $(x,y)\in\mathbb{S}_2$ 
\begin{align*}
w(x,y)=&\frac{1}{\alpha n^2}\left(1-e^{-\alpha n (y-\frac{x-x^\star}{\alpha})}\right)+(x-c)\left(\frac{x-x^\star}{\alpha}\right)-\frac{1}{2\alpha}(x-x^\star)^2\\
=&\frac{1}{\alpha n^2}\left(1-e^{-\alpha n (y-\frac{x-x^\star}{\alpha})}\right)+\left(\frac{x-x^\star}{\alpha}\right)\left[\frac{1}{2}(x-c)+\frac{1}{2}(x^\star-c)\right]
\geq 0,
\end{align*} 
where the last inequality is due to $y>\frac{x-x^\star}{\alpha}$ and $x\geq x^\star\geq c$. Moreover, for $(x,y)\in\mathbb{S}_1$, one obtains 
$$w(x,y)=(x-c)y-\frac{1}{2}\alpha y^2\geq y\Big[x-c-\frac{1}{2}(x-x^\star)\Big]=y\left[\frac{1}{2}(x-c)+\frac{1}{2}(x^\star-c)\right]>0,$$ where we have used $y \leq (x-x^\star)/\alpha$ in the first inequality, and $x\geq x^\star>c$ in the last inequality. Thus, $w$ is nonnegative on $\mathbb{R}\times[0,\infty)$.
\vspace{0.25cm}

\emph{Step 4.} The control $\xi^\star$ given by \eqref{ExtractABM} is admissible, and satisfies \eqref{eq2VT} and \eqref{eq3VT}. Since by \emph{Step 1} and \emph{Step 2} $w$ is a $C^{2,1}$-solution to the HJB equation \eqref{HJB1}, and by \emph{Step 3} satisfies all the requirements of Theorem \ref{VerificationTheorem}, we conclude that $$w(x,y)=V(x,y),\quad(x,y)\in\mathbb{R}\times[0,\infty),$$ by Theorem \ref{VerificationTheorem}. 
\end{proof}

\begin{figure}
	\centering
	\includegraphics[width=0.7\textwidth]{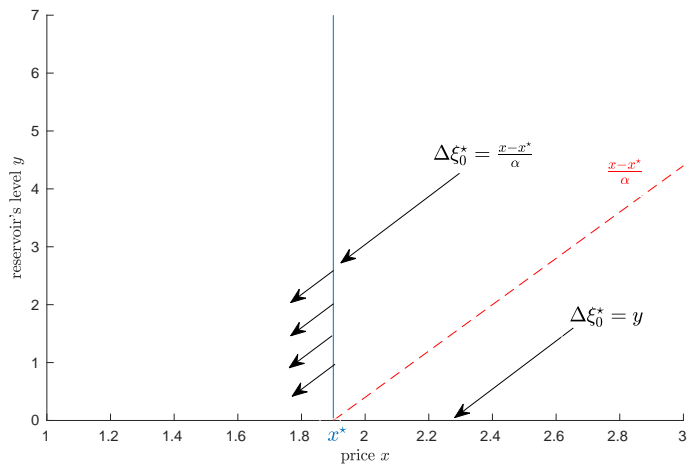}
	\
\caption{A graphical illustration of the optimal extraction rule $\xi^{\star}$ (cf.\ \eqref{ExtractABM}) and of the free boundary $x^{\star}$. The plot has been obtained by using $a=0.4,\, \sigma=0.8,\, \rho=3/8,\, c=0.3,\, \alpha=0.25$. The optimal extraction rule prescribes the following. In the region $\{(x,y) \in \mathbb{R} \times (0,\infty):\, x < x^{\star}\}$ it is optimal not to extract. If at initial time $(x,y)$ is such that $x > x^{\star}$ and $y \leq (x-x^{\star})/\alpha$, then the reservoir should be immediately depleted. On the other hand, if $(x,y)$ is such that $x> x^{\star}$ and $y > (x-x^{\star})/\alpha$, then one should make a lump sum extraction of size $(x-x^{\star})/\alpha$, and then keep on extracting until the commodity is exhausted by just preventing the price to rise above $x^{\star}$.}
\end{figure}

\begin{remark}
Notice that, as $\alpha \downarrow 0$, the optimal extraction rule $\xi^{\star}$ of \eqref{ExtractABM} converges to the extraction rule $\widehat{\xi}$ that prescribes to instantaneously deplete the reservoir as soon as the price reaches $x^{\star}$; i.e., defining, for any given and fixed $(x,y) \in \mathbb{R} \times [0,\infty)$, $\widehat{\tau}(x,y):=\inf\{t\geq0:\, x + a t + \sigma W_t \geq x^{\star}\}$, one has $\widehat{\xi}_t=0$ for all $t<\widehat{\tau}(x,y)$ and $\widehat{\xi}_t=y$ for all $t\geq\widehat{\tau}(x,y)$. The latter control can be easily checked to be optimal for the extraction problem in which the company does not have market impact (i.e.\ $\alpha=0$). 
\end{remark}

%%%%%%%%%%%%%%%%%%%%%%%%%%%%%%%%%%%%%%%%%%%%%%%%%%%%%%%%%%%%%%%%%%%%%%%%%%%%%%%%%%%%%%%%%%%%%%%%%%%%%%%%%%%%%%%%%%%%%%%%%%%%%%%%%%%%%%%%%%%%%%%

\subsection{$b>0$: The Case of a Mean-Reverting Fundamental Price}
\label{sec:OU}

In this section we assume $b>0$, and we study the optimal extraction problem \eqref{ValueFnc} when the commodity's price evolves as a linearly controlled Ornstein-Uhlenbeck process
$$dX_t^{x,\xi}=(a-bX_t^{x,\xi}) dt+\sigma dW_t-\alpha d\xi_t,\quad X^{x,\xi}_{0-} = x\in\mathbb{R},$$ 
for any $\xi\in\mathcal{A}(y)$. Before proceeding with the construction of a candidate optimal solution for \eqref{ValueFnc}, in the next lemma we recall some important properties of the (uncontrolled) Ornstein-Uhlenbeck process that will be needed in our subsequent analysis. Their proof can be found in Appendix \ref{App:A}.

\begin{lemma}
\label{Properties}
	Let $\mathcal{L}$ denote the infinitesimal generator of the uncontrolled Ornstein-Uhlenbeck process (cf.\ \eqref{InfinitesOp}). Then the following hold true.
	\begin{itemize}
		\item [(1)] The strictly increasing fundamental solution to the ordinary differential equation $\cL u-\rho u=0$ is given by
		\begin{align}
		\label{eq34}
		\psi(x)=e^{\frac{(bx-a)^2}{2\sigma^2b}}D_{-\frac{\rho}{b}}\bigg(-\frac{(bx-a)}{\sigma b}\sqrt{2b}\bigg),
		\end{align}
		where 
		\begin{align}\label{eq35}
		D_\beta(x):=\frac{e^{-\frac{x^2}{4}}}{\Gamma(-\beta)}\int_{0}^{\infty}t^{-\beta-1}e^{-\frac{t^2}{2}-xt}dt,\quad \beta<0,
		\end{align} 
		is the Cylinder function of order $\beta$ and $\Gamma(\,\cdot\,)$ is the Euler's Gamma function (see, e.g., Chapter VIII in \cite{bateman}). Moreover, $\psi$ is strictly convex.
		
		\item [(2)] Denoting by $\psi^{(k)}$ the k-th derivative of $\psi$, $k\in\mathbb{N}$, one has that $\psi^{(k)}$ is strictly convex and it is (up to a constant) the positive strictly increasing fundamental solution to $(\cL-(\rho+kb))u=0$.
		
		\item [(3)] For any $k\in\mathbb{N} \cup \{0\}$, $\psi^{(k+2)}(x)\psi^{(k)}(x)-\psi^{(k+1)}(x)^2>0$ for all $x\in\mathbb{R}$.
	\end{itemize}
\end{lemma}

For any $y>0$, from \eqref{cond4} we find a representation of $A(y)$ in terms of $G(y)$; that is,
\begin{align}
\label{A(y)}
A(y)=\frac{(G(y)-c)\psi'(G(y))-\psi(G(y))}{\alpha[\psi'(G(y))^2-\psi''(G(y))\psi(G(y))]}.
\end{align}
Notice that the denominator of $A(y)$ is nonzero due to Lemma \ref{Properties}-(3). 

For our subsequent analysis it is convenient to look at $G$ as a function of the state variable $y \in (0,\infty)$, and, in particular, we conjecture that it is the inverse of an injective nonnegative function $F$ to be endogenously determined together with its domain and its behavior. This is what we are going to do in the following. From now on we set $G\equiv F^{-1}$.

Since we have $V(x,0)=0$ (cf.\ Theorem \ref{VerificationTheorem}) for any $x \in \mathbb{R}$, we impose $A(0)=0$. Then, from \eqref{A(y)} we obtain the boundary condition 
\begin{equation}
\label{def:x0}
x_0:=F^{-1}(0)\quad \text{solving} \quad (x_0-c){\psi'(x_0)}-{\psi(x_0)}=0.
\end{equation}
In fact, existence and uniqueness of such $x_0$ is given by the following (more general) result. Its proof can be found in Appendix \ref{App:A}.

\begin{lemma}
\label{UniqueSol}
Recall that $\psi^{(k)}$ denotes the derivative of order $k$, $k\in\mathbb{N} \cup \{0\}$, of $\psi$. Then, for any $k\in\mathbb{N} \cup \{0\}$, there exists a unique solution on $(c,\infty)$ to the equation $(x-c){\psi^{(k+1)}(x)}-{\psi^{(k)}(x)}=0$. In particular, there exists $x_0>c$ uniquely solving $(x-c){\psi'(x)}-{\psi(x)}=0$ and $x_{\infty}>c$ uniquely solving $(x-c){\psi''(x)}-{\psi'(x)}=0$.
\end{lemma}

From  \eqref{eq1} and \eqref{eq2} we have
\begin{align}
\label{DA}
A'(y)=\frac{(F^{-1}(y)-c)\psi''(F^{-1}(y))-\psi'(F^{-1}(y))}{\psi''(F^{-1}(y))\psi(F^{-1}(y))-\psi'(F^{-1}(y))^2}, \quad y>0,
\end{align}
and the denominator of $A'(y)$ is nonzero due to Lemma \ref{Properties}-(3). 

Now, we define the functions $M:\mathbb{R}\mapsto\mathbb{R}$ and $N:\mathbb{R}\mapsto\mathbb{R}$ such that for any $x\in\mathbb{R}$
\begin{align}
\label{FBFunc}
M(x):=\frac{(x-c)\psi'(x)-\psi(x)}{\alpha[\psi'(x)^2-\psi''(x)\psi(x)]},\quad
N(x):=\frac{(x-c)\psi''(x)-\psi'(x)}{\psi''(x)\psi(x)-\psi'(x)^2},
\end{align}
and, by differentiating $M$ and rearranging terms, we obtain
$$M'(x)=\frac{\left[\psi'''(x)\left[(x-c)\psi'(x)-\psi(x)\right]-\psi''(x)\left[(x-c)\psi''(x)-\psi'(x)\right]\right]\psi(x)}{\alpha[\psi'(x)^2-\psi''(x)\psi(x)]^2}.$$

However, by noticing that $M(x)=A(F(x))$ (cf.\ \eqref{A(y)} and \eqref{FBFunc}), the chain rule yields $M'(x)=A'(F(x))F'(x),$ which in turn gives
\begin{align}\label{eq6}
F'(x)=\frac{M'(x)}{N(x)},
\end{align}
upon observing that $N(x)=A^\prime(F(x))$ from \eqref{DA} and \eqref{FBFunc}.

Recall that by Lemma \ref{UniqueSol} there exists a unique $x_\infty>c$ solving $N(x_\infty)=0$; that is, solving $(x-c)\psi''(x)-\psi'(x)=0$. Due to \eqref{eq6}, this point is a vertical asymptote of $F'$, and the next result shows that $x_\infty$ is located to the left of $x_0$. The proof can be found in Appendix \ref{App:A}.

\begin{lemma}
\label{xinfty}
Recall Lemma \ref{UniqueSol} and let $x_0$ and  $x_\infty$ be the unique solutions to $M(x)=0$ (i.e.\ $(x-c)\psi'(x)-\psi(x)=0$) and $N(x)=0$ (i.e.\ $(x-c)\psi''(x)-\psi'(x)=0$), respectively. We have $x_{\infty}<x_0$.
\end{lemma}

The following useful corollary immediately follows from the proof of Lemma \ref{UniqueSol}.
\begin{corollary}\label{Order}
One has
$$(x-c)\psi'(x)-\psi(x)< 0,\quad \text{for all }x< x_0,$$
and
$$(x-c)\psi''(x)-\psi'(x)> 0,\quad \text{for all }x> x_\infty.$$
\end{corollary}

By integrating \eqref{eq6} in the interval $[x,x_0]$, for $x\in(x_\infty,x_0]$, and using the fact that $F(x_0)=0$ (cf.\ \eqref{def:x0}), we obtain 
\begin{align}\label{boundary}
F(x)=\int_{x}^{x_0}\frac{\left[\psi'''(x)\left[(x-c)\psi'(x)-\psi(x)\right]-\psi''(x)\left[(x-c)\psi''(x)-\psi'(x)\right]\right]\psi(x)}{-\alpha[\psi''(z)\psi(z)-\psi'(z)^2][(z-c)\psi''(z)-\psi'(z)]}dz,
\end{align}
which is well defined, but possibly infinite for $x=x_\infty$. In the following we will refer to $F$ as to the \emph{free boundary}. We now prove properties of $F$ that have been only conjectured so far.

\begin{proposition}
\label{propertiesIntegrand}
The free boundary $F$ defined in \eqref{boundary} is strictly decreasing for all $x\in(x_\infty,x_0)$ and belongs to $C^{\infty}((x_\infty,x_0])$. Moreover, 
\begin{equation}
\label{LimitsF}
\lim_{x\downarrow x_\infty}F(x)=\infty = \lim_{x\downarrow x_\infty}F'(x).
\end{equation}
\end{proposition}

\begin{proof}
 \emph{Step 1.} We start by proving the claimed monotonicity. Notice that by \eqref{boundary} one has $F'(z)=-\Theta(z)$, where the function $\Theta:(x_\infty,\infty]\mapsto\mathbb{R}$ is given by 
	$$\Theta(z):=\frac{\left[\psi'''(z)\left[(z-c)\psi'(z)-\psi(z)\right]-\psi''(z)\left[(z-c)\psi''(z)-\psi'(z)\right]\right]\psi(z)}{-\alpha[\psi''(z)\psi(z)-\psi'(z)^2][(z-c)\psi''(z)-\psi'(z)]}.$$ 
	By Lemma \ref{Properties} one has $\psi''(z)\psi(z)-\psi'(z)^2>0$ for any $z\in\mathbb{R}$. Moreover, $\Phi(z):=(z-c)\psi''(z)-\psi'(z)>0$  for all $z>x_\infty>c$ by Corollary \ref{Order}. Therefore the denominator of $\Theta$ is strictly negative for any $z\in(x_\infty,x_0)$. Again, an application of Corollary \ref{Order} implies that the numerator of $\Theta$ is strictly negative for any $z \in (x_\infty,x_0)$, and therefore $\Theta>0$ and $F'<0$. Thus, we conclude that $F$ is strictly decreasing.´
	\vspace{0.25cm}

	\emph{Step 2.} To prove \eqref{LimitsF}, recall that from \emph{Step 1} we have set $\Phi(z)=(z-c)\psi^{\prime\prime}(z)-\psi^\prime(z)>0$ for all $z \in (x_{\infty},x_0)$, and define 
	$$h(z):=\frac{\left[\psi'''(z)\left[(z-c)\psi'(z)-\psi(z)\right]-\psi''(z)\left[(z-c)\psi''(z)-\psi'(z)\right]\right]\psi(z)}{-\alpha[\psi''(z)\psi(z)-\psi'(z)^2]},\quad z\in(x_\infty,x_0),$$ which is continuous and nonnegative by \emph{Step 1}. Notice that $h/\Phi=\Theta$, with $\Theta$ as in \emph{Step 1}. 
	
	By de l'Hopital's rule,
	$$\lim_{z\downarrow x_\infty}\frac{\Phi(z)}{z-x_\infty}=\lim_{z\downarrow x_\infty}\Phi^\prime(z)=(x_\infty-c)\psi^{\prime\prime\prime}(x_\infty)=:\ell>0,$$ 
	so that, for any $\varepsilon>0$, there exists $\delta_\varepsilon>0$ such that if $|z-x_\infty|<\delta_\varepsilon$, then $\big|\frac{\Phi(z)}{z-x_\infty}-\ell\big|<\varepsilon$. 
	Thus, for any $\varepsilon>0$, we let $\delta_\varepsilon$ be as above, and we take $x \in (x_{\infty},x_{\infty}+\delta_{\varepsilon})$. Then, recalling \eqref{boundary}, we see that there exists a constant $C>0$ (possibly depending on $x_\infty$ and $x_0$, but not on $x$) such that
	\begin{align*}
	& F(x)= \int_{x}^{x_0} \Theta(z) dz = \int_{x}^{x_0}\frac{h(z)}{(z-x_\infty)\frac{\Phi(z)}{(z-x_\infty)}}dz \nonumber \\
	& \geq \int_{x}^{x_\infty+\delta_\varepsilon}\frac{C}{(\ell+\varepsilon)}\frac{dz}{(z-x_\infty)}+C\int_{x_\infty+\delta_\varepsilon}^{x_0}\frac{dz}{\Phi(z)} \rightarrow \infty
	\end{align*}
	as $x \downarrow x_\infty$.

Finally, since the integrand in \eqref{boundary} is a $C^{\infty}$-function on $(x_\infty,x_0]$, it follows that $F$ is so as well.
\end{proof}

\begin{remark}
	\label{x0xinfty}
	The critical price levels $x_0$ and $x_{\infty}$ have a clear interpretation. $x_0$ is the free boundary arising in the optimal extraction problem when we set $\alpha=0$, so that the company's actions have no market impact. $x_{\infty}$ is the free boundary of the optimal extraction problem when there is an infinite amount of commodity available in the reservoir, i.e.\ $y=\infty$.
\end{remark}

Given $F$ as above, we now introduce the sets $\mathbb{S}_1$ and $\mathbb{S}_2$ that partition the (candidate) \emph{selling region} $\mathbb{S}$:
\begin{align*}
\mathbb{S}_1&:=\{(x,y)\in\mathbb{R}\times(0,\infty)\,:\, x \geq F^{-1}(y)\text{ and }y\leq (x-x_0)/{\alpha} \},\\
\mathbb{S}_2&:=\{(x,y)\in\mathbb{R}\times(0,\infty)\,:\, x \geq F^{-1}(y)\text{ and }y> (x-x_0)/{\alpha} \}.
\end{align*}
and the (candidate) \emph{waiting region}
$$\mathbb{W}:=\{(x,y)\in\mathbb{R}\times(0,\infty)\,:\, x< F^{-1}(y)\} \cup (\mathbb{R} \times \{0\}).$$

We now make a guess on the structure of the optimal strategy in terms of the sets $\mathbb{W}$ and $\mathbb{S}_1$ and $\mathbb{S}_2$. If the current price $x$ is sufficiently low, and in particular it is such that $x<F^{-1}(y)$ (i.e. $(x,y)\in\mathbb{W}$), we conjecture that the company does not extract, and the payoff accrued is just the continuation value $A(y)\psi(x)$. Whenever the price attempts to cross the critical level $F^{-1}(y)$, then the company makes infinitesimal extractions that keep the state process $(X,Y)$ inside the region $\{(x,y)\in\mathbb{R}\times(0,\infty):\,x\leq F^{-1}(y)\}$. If the current price $x$ is sufficiently high (i.e.\ $x>F^{-1}(y)$) and the current level of the reservoir is sufficiently large (i.e. lies in $\mathbb{S}_2$), then the company makes an instantaneous lump sum extraction of suitable amplitude $z$, and pushes the joint process $(X,Y)$ to the locus of points $\{(x,y)\in\mathbb{R}\times(0,\infty): y=F(x)\}$, and then continues extracting as before. The associated payoff is then the sum of the continuation value starting from the new state $(x-\alpha z,y-z)$, and the profits accrued from selling $z$ units of the commodity, that is $(x-c)z-\frac{1}{2}\alpha z^2$. If the current capacity level is not large enough (i.e.\ $y\leq\frac{x-x_0}{\alpha}$, so that $(x,y)\in\mathbb{S}_1$), then the company immediately depletes the reservoir. This action is associated to the net profit $(x-c)y-\frac{1}{2}\alpha y^2$.

In light of the previous conjecture we therefore define our candidate value function as
\begin{align}
\label{CaFnc}
w(x,y):=
\begin{cases}
A(y)\psi(x),\quad &\text{if } (x,y) \in \mathbb{W},\\
A\big(F(x-\alpha z)\big)\psi(x-\alpha z)+(x-c)z-\frac{1}{2}\alpha z^2,&\text{if } (x,y) \in \mathbb{S}_2,\\
(x-c)y-\frac{1}{2}\alpha y^2, &\text{if } (x,y) \in \mathbb{S}_1,\\
\end{cases}
\end{align}
where, for any $(x,y)\in\mathbb{S}_2$, we denote by $z:=z(x,y)$ the unique solution to 
\begin{align}
\label{optimalExercise}
y-z=F(x-\alpha z).
\end{align}
In fact, its existence and uniqueness is guaranteed by the next lemma, whose proof is in Appendix \ref{App:A}.
\begin{lemma}
\label{Existencez}
For any $(x,y)\in\mathbb{S}_2$, there exists a unique solution $z(x,y)$ to \eqref{optimalExercise}. Moreover, we have $z(x,y) \in (\frac{x-x_0}{\alpha},\frac{x-x_\infty}{\alpha} \wedge y]$,
\begin{align}\label{eq54}
z(x,F(x))=0 \quad \text{for any } x\in(x_\infty,x_0),
\end{align} and
\begin{align}\label{eq55}
z(x,y)=\frac{x-x_0}{\alpha},\quad\text{for any $(x,y)\in\mathbb{R}\times(0,\infty)$ such that $x\geq x_0$ and $y=\frac{x-x_0}{\alpha}$.}
\end{align}
\end{lemma}

Next, we verify that $w$ is a classical solution to the HJB equation  \eqref{HJB1}. This is accomplished in the next two results.

\begin{lemma}
\label{Smooth}
The function $w$ is $C^{2,1}(\mathbb{R}\times [0,\infty))$.
\end{lemma}

\begin{proof}
Continuity is clear by construction. We therefore need to eveluate the derivatives of $w$.

Denoting by $\I(\cdot)$ the interior of a set, we have by \eqref{CaFnc} that for all $(x,y)\in \I(\mathbb{W})$ 
	\begin{align}\label{DerWait}
		w_x(x,y)=A(y)\psi^\prime(x),\quad
		w_{xx}(x,y)=A(y)\psi^{\prime\prime}(x),\quad
		w_y(x,y)=A^\prime(y)\psi(x),
	\end{align}
	 and that for all $(x,y)\in \I(\mathbb{S}_1)$
	\begin{align}\label{DerS1}
		w_x(x,y)=y,\quad
		w_{xx}(x,y)=0,\quad
		w_y(x,y)=x-c-\alpha y.
	\end{align}
The previous equations easily give the continuity of the derivatives in $\I(\mathbb{W})$, $\I(\mathbb{S}_1)$ and in $\mathbb{R} \times \{0\}$.

To evaluate $w_x$, $w_{xx}$ and $w_y$ for $(x,y)\in \I(\mathbb{S}_2)$, we need some more work.  From \eqref{optimalExercise}, we calculate the derivatives of $z=z(x,y)$ with respect to $x$ and $y$ by the help of the implicit function theorem, and we obtain
	\begin{align}\label{eq20}
	z_x(x,y)=\frac{F'(x-\alpha z)}{\alpha F'(x-\alpha z)-1},
	\end{align}
	and
	\begin{align}\label{eq21}
	z_y(x,y)=\frac{1}{1-\alpha F'(x-\alpha z)},
	\end{align}
	for any $(x,y)\in \I(\mathbb{S}_2)$.
	Moreover, recalling that we have set $G\equiv F^{-1}$, and taking $y=F(x-\alpha z)$, we find from \eqref{eq1}
	\begin{align}\label{eq22}
	A'(F(x-\alpha z))=\frac{x-\alpha z-c}{\psi(x-\alpha z)}-\alpha A(F(x-\alpha z))\frac{\psi'(x-\alpha z)}{\psi(x-\alpha z)},
	\end{align}
	and from \eqref{eq2}
	\begin{align}
	\label{eq23}
	A'(F(x-\alpha z))=\frac{1-\alpha A(F(x-\alpha z))\psi''(x-\alpha z)}{\psi'(x-\alpha z)}.
	\end{align}
	
  By differentiating $w$ with respect to $x$ strictly inside $\mathbb{S}_2$ (cf.\ the second line of \eqref{CaFnc}), and using \eqref{eq20} and \eqref{eq22}, we obtain
	\begin{align}
	\label{eq24}
	w_x(x,y)=A(F(x-\alpha z))\psi^\prime(x-\alpha z)+z.
	\end{align}
	Also, by \eqref{eq23} and \eqref{eq20}
	\begin{align}
	\label{eq25}
	w_{xx}(x,y)=A(F(x-\alpha z))\psi^{\prime\prime}(x-\alpha z).
	\end{align}
	Moreover, differentiating with respect to $y$ the second line of \eqref{CaFnc}, and using \eqref{eq21} and \eqref{eq22}, yields
	\begin{align}
	\label{eq26}
		w_y(x,y)=A^\prime(F(x-\alpha z))\psi(x-\alpha z).
	\end{align}
Equations \eqref{eq24}-\eqref{eq26} hold for any $(x,y)\in \I(\mathbb{S}_2)$, and give that $w \in C^{1,2}(\I(\mathbb{S}_2))$.

Now, let $(x_n,y_n)_n\subseteq \I(\mathbb{S}_2)$ be any sequence converging to $(x,F(x))$, $x\in(x_\infty,x_0]$. Since $\lim_{n\rightarrow\infty}z(x_n,y_n)=0$ by continuity of $z$, and because $A$, $\psi$, $\psi'$ and $\psi''$ are also continuous, we conclude from \eqref{DerWait} and \eqref{eq24}--\eqref{eq26} that $w\in C^{2,1}(\overline{\mathbb{W}}\cap\overline{\mathbb{S}_2})$, where $\overline{\mathbb{W}}$ and $\overline{\mathbb{S}_2}$ denote the closures of $\mathbb{W}$ and $\mathbb{S}_2$.
 
In order to prove that $w\in C^{2,1}(\mathbb{S}_1\cap\overline{\mathbb{S}_2})$, consider a sequence $(x_n,y_n)_n \subseteq \mathbb{S}_2$ converging to $(x,\frac{x-x_0}{\alpha})$, $x\geq x_0$. Again by the continuity of $F$ and exploiting that $F(x_0)=0$ we get $\lim\limits_{n\rightarrow\infty}z(x_n,y_n)=\frac{1}{\alpha}(x-x_0)$. Therefore, we have $w\in C^{2,1}(\mathbb{S}_1\cap\overline{\mathbb{S}_2})$ by \eqref{DerS1} and \eqref{eq24}--\eqref{eq26}, and upon employing $A(F(0))=0$ and $\psi(x_0)A'(F(0))=\frac{\psi(x_0)}{\psi'(x_0)}=x_0-c$ by \eqref{eq23}.

Collecting all the previous results, the claim follows.
\end{proof}

\begin{proposition}
\label{HJB}
The function $w$ as in \eqref{CaFnc} is a $C^{2,1}(\mathbb{R}\times [0,\infty))$ solution to the HJB equation \eqref{HJB1}, and it is such that $w(x,0)=0$. 
\end{proposition}

\begin{proof}
The claimed regularity follows from Lemma \ref{Smooth}, whereas we see from \eqref{CaFnc} that $w(x,0)=0$ since $A(0)=0$. Hence, we assume in the following that $y>0$. Moreover, it is important to recall that in \eqref{eq1} and \eqref{eq2} we have set $G\equiv F^{-1}$. 

By construction $\mathcal{L}w(x,y)-\rho w(x,y)=0$ for all $(x,y)\in\mathbb{W}$. Moreover, $-\alpha w_x(x,y)-w_y(x,y)+(x-c)=0$ for all $(x,y)\in\mathbb{S}_1$. Also, $-\alpha w_x(x,y)-w_y(x,y)+(x-c)=0$ for all $(x,y)\in\mathbb{S}_2$ by employing \eqref{eq24} and \eqref{eq26}, and observing that from \eqref{eq1} one has $$-\alpha A(F(x-\alpha z))\psi'(x-\alpha z)-A'(F(x-\alpha z))\psi(x-\alpha z)+(x-\alpha z)-c=0.$$ 
	Hence, it is left to show that
	\begin{align}\label{cond5}	
	 -\alpha w_{x}(x,y)-w_y(x,y)+x-c&\leq 0,\quad\forall (x,y)\in\mathbb{W},\\
	 \label{cond6}
	\mathcal{L}w(x,y)-\rho w(x,y)&\leq 0,\quad\forall (x,y)\in\mathbb{S}=\mathbb{S}_1\cup \mathbb{S}_2
	\end{align} 
	In \emph{Step 1} below we prove that \eqref{cond5} holds, whereas the proof of \eqref{cond6} is separately performed for $\mathbb{S}_1$ and $\mathbb{S}_2$ in \emph{Step 2} and \emph{Step 3} respectively.
	\vspace{0.25cm}
	
	\emph{Step 1.} Here we prove that \eqref{cond5} holds for any $(x,y)\in\mathbb{W}$. Notice that \eqref{eq1} gives
	\begin{align}\label{eq27}
	A'(y)=\frac{F^{-1}(y)-c}{\psi(F^{-1}(y))}-\frac{\alpha A(y)\psi'(F^{-1}(y))}{\psi(F^{-1}(y))}.
	\end{align}
	Then, by using the first and the third equation of \eqref{DerWait}, and \eqref{eq27}, we rewrite the left-hand side of \eqref{cond5} (after rearranging terms) as 
	\begin{align}\label{eq56}
	\alpha A(y)\bigg[\frac{\psi^\prime(F^{-1}(y))\psi(x)}{\psi(F^{-1}(y))}-\psi^\prime(x)\bigg]-\frac{F^{-1}(y)-c}{\psi(F^{-1}(y))}\psi(x)+x-c=Q(x,F^{-1}(y)),
	\end{align}
	for any $(x,y)\in\mathbb{W}$. 
	Here, we have defined 
  $$Q(x,q):=\alpha A(F(q))\bigg[\frac{\psi^\prime(q)\psi(x)}{\psi(q)}-\psi^\prime(x)\bigg]-\frac{q-c}{\psi(q)}\psi(x)+x-c,$$
	for any $(x,q)\in\mathbb{R} \times [x_{\infty},x_0]$. Since $Q(q,q)=0$, in order to have \eqref{cond5} it suffices to show that one has (recall that $(x_\infty,x_0]$ is the domain of $F$) 
$$Q_x(x,q)\geq 0,\quad\text{for any $x\leq q,$ for all $q\in(x_\infty,x_0]$}.$$ 
We prove this in the following.
	
	Differentiating $Q$ with respect to $x$, and using \eqref{A(y)}, gives
	\begin{align}
	\label{eq31}
Q_x(x,q)=\frac{\psi(q)-(q-c)\psi^\prime(q)}{\psi^{\prime\prime}(q)\psi(q)-\psi^\prime(q)^2}\bigg[\frac{\psi^\prime(x)\psi^\prime(q)}{\psi(q)}-\psi^{\prime\prime}(x)\bigg]-(q-c)\frac{\psi^\prime(x)}{\psi(q)}+1.
	\end{align} 
	
	Take $x\leq x_\infty$ and $q=x_\infty$, and recall that $x_\infty>c$ solves $(x_\infty-c)=\frac{\psi'(x_\infty)}{\psi''(x_\infty)}$. Then, after some simple algebra,  we have $$Q_x(x,x_\infty)=1-\frac{\psi^{\prime\prime}(x)}{\psi^{\prime\prime}(x_\infty)}>0,$$ where the last inequality is due to the fact that $x\mapsto\psi''(x)$ is strictly increasing.
	
	Moreover, we find
	\begin{align}
	\label{eq5}
	Q_x(x,x_0)=1-(x_0-c)\frac{\psi'(x)}{\psi(x_0)}\geq 0,\quad\text{for any $x\leq x_0,$}
	\end{align}
	due to the fact that $x_0>c$ uniquely solves $(x_0-c)\psi'(x_0)-\psi(x_0)=0$ and $x\mapsto 1-(x_0-c)\frac{\psi'(x)}{\psi(x_0)}<0$ is strictly decreasing.
	
	By differentiating $Q_x$ of \eqref{eq31} with respect to $q$ one obtains 
	\begin{align}
	\label{Qxq}
	Q_{xq}(x,q)=\left[\frac{\psi'''(q)\left[(q-c)\psi'(q)-\psi(q)\right]-\psi''(q)\left[(q-c)\psi''(q)-\psi'(q)\right]}{\big(\psi^{\prime\prime}(q)\psi(q)-\psi^\prime(q)^2\big)^2}\right]\Phi(x,q),
	\end{align}
	where we have introduced the function 
	\begin{align}
	\label{Phi}
	\Phi(x,q):=\psi^\prime(x)\psi^\prime(q)-\psi^{\prime\prime}(x)\psi(q),\quad\text{for all $(x,q)\in\mathbb{R}^2$},
	\end{align} 
	that is such that
	\begin{align}\label{eq4}
	\Phi_q(x,q)=\psi^\prime(x)\psi^{\prime\prime}(q)-\psi^{\prime\prime}(x)\psi^\prime(q)>0,\quad\forall x\leq q,
	\end{align}
	since ${\psi^\prime}/{\psi^{\prime\prime}}$ is decreasing due to Lemma \ref{Properties} with $k=1$. 
	
	 By Corollary \ref{Order} we have that 
	 \begin{align}\label{cond7}
	 \psi'''(q)\left[(q-c)\psi'(q)-\psi(q)\right]-\psi''(q)\left[(q-c)\psi''(q)-\psi'(q)\right]\leq 0,
	 \end{align} 
	 for all $q\in[x_\infty,x_0]$. Hence, the term multiplying $\Phi$ in the right-hand side of \eqref{Qxq} is negative. 
	 
	 In light of \eqref{eq4}, we know that $\Phi(x,q)$ is increasing in $q$ for $q\geq x$. We now have three possible cases. 
	 
	 (a) If $\Phi$ is such that $\Phi(x,q)<0$ for all $q\in[x_\infty,x_0]$, then by \eqref{cond7} (and noticing that the function in \eqref{cond7} in fact appears in the numerator of $Q_{xq}$) we must have $Q_{xq}(x,q)\geq 0$ for all $q\in[x_\infty,x_0]$, so that
	\begin{align}
	\label{eq28}
	 0\leq Q_x(x,x_\infty)\leq Q_x(x,q)\leq Q_x(x,x_0),\quad\text{for all $q\in[x_\infty,x_0]$, and $x\leq x_\infty$}.
	 \end{align} 
	
	 (b) If $\Phi$ is such that $\Phi(x,q)>0$ for all $q\in[x_\infty,x_0]$, then by \eqref{cond7} we must have $Q_{xq}(x,q)\leq 0$ for all $q\in[x_\infty,x_0]$, so that 
	 \begin{align}
	 \label{eq29}
	 0\leq Q_x(x,x_0)\leq Q_x(x,q)\leq Q_x(x,x_\infty),\quad\text{for all $q\in[x_\infty,x_0]$, and $x\leq x_\infty$.}
	 \end{align}

	 (c) If $\Phi$ is such that $\Phi(x,q)\leq 0$ for all $q\in[x_\infty,\bar{q}]$, where $\bar{q}\in[x_\infty,x_0]$, and $\Phi(x,q)>0$ for all $q\in[\bar{q},x_0]$, then by \eqref{cond7} we must have $Q_{xq}(x,q)\geq 0$ for all $q\in[x_\infty,\bar{q}]$, and $Q_{xq}(x,q)\leq 0$ for all $q\in[\bar{q},x_0]$, so that 
	 \begin{align}
	 \label{eq30}
	 Q_x(x,q)\geq\min\{Q_x(x,x_\infty),Q_x(x,x_0)\}\geq 0,\quad\text{for all $q\in[x_\infty,x_0]$ and $x \leq x_\infty$.}
	 \end{align}
From \eqref{eq28}-\eqref{eq30}, we then conclude that \eqref{cond5} holds for any $(x,y)\in\mathbb{W}$ such that $x\leq x_\infty$.
	 	 
	 Now, take $x\in(x_\infty,x_0]$ and let $q\in[x,x_0]$. For $q=x$ we find from \eqref{eq31} that
	 \begin{align}\label{eq32}
	 Q_x(x,x)=0.
	 \end{align}
	 Then, proceeding as above, from \eqref{eq5} and \eqref{eq32}, we obtain that $Q_x(x,q)\geq 0$ for all $x\in(x_\infty,x_0]$ with $q\in[x,x_0]$. 
	 
	 Hence, in conclusion, $Q_x(x,F^{-1}(y))\geq 0$ for all $x\leq F^{-1}(y)$ and $y>0$, and \eqref{cond5} is then established.
	 \vspace{0.25cm}
	 
	 \emph{Step 2.} Here, we show that \eqref{cond6} holds in $\mathbb{S}_1$. Setting 
	$$\bar{x}=\frac{a+\rho c}{\rho + b},$$ 
	by Lemma \ref{Propx_0xbar} in Appendix \ref{App:B} we have $\bar{x}\leq x_0$, with $x_0$ solving $(x_0-c)\psi'(x_0)-\psi(x_0)=0$ (cf.\ Lemma \ref{UniqueSol}).
	 
	 Now, let $(x,y)\in\mathbb{S}_1$ be given and fixed. Thanks to the first and second equation in  \eqref{DerS1} we have
	 \begin{align*}
	 \mathcal{L}w(x,y)-\rho w(x,y)=(a-bx)y-\rho\Big[(x-c)y-\frac{1}{2}\alpha y^2\Big]=:\widetilde{Q}(x,y).
	 \end{align*}
	 Clearly $\widetilde{Q}(x,0)=0$. Also, since $(x,y)\in\mathbb{S}_1$ is such that $y\leq \frac{1}{\alpha}(x-x_0)$ and $x\geq x_0$, we have
	 \begin{align*}
	 \widetilde{Q}_y(x,y)=a-bx-\rho(x-c)+\alpha\rho y\leq a-bx-\rho(x_0-c)\leq a+\rho c-x_0(\rho + b)\leq 0,
	 \end{align*}
	 where the last inequality is due to $x_0\geq\bar{x}$. Hence $\mathcal{L}w(x,y)-\rho w(x,y)\leq 0$ on $\mathbb{S}_1$.
	 
	  \vspace{0.25cm}
	 
	 \emph{Step 3.} Here we provide the proof of \eqref{cond6} in $\mathbb{S}_2$, separately for the two cases: \emph{(i)} $a-bc \leq 0$ and \emph{(ii)} $a-bc>0$, and different approaches are followed in these two cases (see also Remark \ref{rem:verificoBech} below).
	 
	 	  \vspace{0.15cm} 
	 	  
	 	  \emph{(i)} Assume $a-bc \leq 0$. Let $(x,y)\in\mathbb{S}_2$ be given and fixed, and recall that $x\geq F^{-1}(y)$ and $y> \frac{1}{\alpha}(x-x_0)$ for all $(x,y)\in\mathbb{S}_2$. By employing \eqref{eq24} and \eqref{eq25}, and observing that from \eqref{cond1} one has 
	 \begin{align}
	 \label{eq33}
	 \begin{split}
	 &\Big[\frac{\sigma^2}{2} A(F(x-\alpha z))\psi''(x-\alpha z)+\big(a-b(x-\alpha z)\big)A(F(x-\alpha z))\psi'(x-\alpha z) \\
	 & \hspace{0.25cm} -\rho A(F(x-\alpha z))\psi(x-\alpha z)\Big]\Big|_{z=z(x,y)} =0,
	 \end{split}
	 \end{align}
	 we get
	 \begin{align}
	 \label{widebarQ}
	 \mathcal{L}w(x,y)-\rho w(x,y)=\Big[(a-bx)z-\rho(x-c)z+\frac{1}{2}\rho\alpha z^2 - b\alpha z A(F(x-\alpha z))\psi'(x-\alpha z)\Big]\Big|_{z=z(x,y)}.
	 \end{align}
	 Since $z>0$, $A>0$, and $\psi'>0$, one has that $\mathcal{L}w(x,y)-\rho w(x,y) \leq \widehat{Q}(x,y)$, where we have set
	 $$\widehat{Q}(x,y):=\Big[(a-bx)z-\rho(x-c)z+\frac{1}{2}\rho\alpha z^2\Big]\Big|_{z=z(x,y)}.$$
	 
	 Observe that $\widehat{Q}(F^{-1}(y),y)=0$ since $z(F^{-1}(y),y)=0$ (cf.\ \eqref{eq54}). Hence, it suffices to show that $\widehat{Q}_x(x,y)<0$ for all $(x,y)\in\mathbb{S}_2$. Differentiating $\widehat{Q}$ with respect to $x$ gives
	 \begin{align*}
	 \widehat{Q}_x(x,y)=z(x,y)\Big(-b-\rho+\rho\alpha z_x(x,y)\Big)+z_x(x,y)\Big[(a-bx)-\rho(x-c)\Big].
	 \end{align*} 
	 
	 Since $z_x>0$ and $\alpha z_x<1$ (cf.\ \eqref{eq20} and recall that $F^\prime<0$), and $x\geq F^{-1}(y)\geq x_\infty$, we find 
	 \begin{align}\label{widebarQx}
	 \begin{split}
	 \widehat{Q}_x(x,y)&\leq z_x(x,y)\Big[a+\rho c-F^{-1}(y)(\rho + b)\Big]\\
	 &\leq z_x(x,y)\big[a+\rho c-x_\infty(\rho + b)\big]=z_x(x,y)(\rho + b)\big(\bar{x}-x_\infty\big),
	 \end{split}	 
	 \end{align} 
	 and clearly $\widehat{Q}_x(x,y)\leq 0$ if $a-bc \leq 0$, since the latter implies $\bar{x} \leq c < x_\infty$. 
	 
	 	 This shows that $\widehat{Q}<0$ on $\mathbb{S}_2$, and therefore that $w$ solves \eqref{cond6} in $\mathbb{S}_2$ if $a-bc\leq0$.
	 	 \vspace{0.15cm}
	 	 
	 	 \emph{(ii)} Assume that $a-bc>0$. In this case, as discussed in Remark \ref{rem:verificoBech}, we did not succeed proving \eqref{cond6} by studying the sign of $\mathcal{L}w-\rho w$ as done in \emph{(i)} above. Therefore, we follow a different approach which is based on that developed in the proof of Lemma 6.7 in \cite{Becherer}. Here we just provide the main ideas, since most of the arguments follow from \cite{Becherer}.
	 	 
Let $(x,y)\in\overline{\mathbb{W}}\cap\mathbb{S}_2$ be given and fixed, and consider an arbitrary $z_o>0$. From \eqref{optimalExercise} we find $z(x+\alpha z_o,y+z_o)=z_o$, and employing the latter we have from \eqref{CaFnc}, \eqref{eq24} and \eqref{eq25} that
	  \begin{align*}
	  &\mathcal{L}w(x+\alpha z_o,y+z_o)-\rho w(x+\alpha z_o,y+z_o)\\
	  &=-\alpha b z_o A(F(x))\psi'(x)+\big(a-b(x+\alpha z_o)\big)z_o-\rho\big((x+\alpha z_o)-c\big)z_o+\frac{1}{2}\rho\alpha z_o^2=:U(z_o).
	  \end{align*} 
	  Notice that $U(0)=0$, hence to show negativity of $U$ it suffices to prove that $U'(z_o)\leq 0$ for all $z_o > 0$. We find
	  \begin{align}
	  \label{Uprime}
	  \begin{split}
	  U'(z_o)&=-\alpha bA(F(x))\psi'(x)-\alpha bz_o+(a-b(x+\alpha z_o))-\rho(x+\alpha z_o-c)\\
	  &=b\left(x-c-\alpha A(F(x))\psi'(x)\right)+(x+\alpha z_o-c)\left[-(b+\rho)+\frac{a-b(x+\alpha z_o)}{(x+\alpha z_o)-c}\right],
	  \end{split}	  
	  \end{align}
	 after rearranging terms, and adding and substracting the term $b(x-c)$ to obtain the second equality above. Now, define the function  
	\begin{equation}
	\label{eq:kappa}
	\kappa(x):=-(b+\rho)+\frac{a-bx}{x-c},
	\end{equation}
	and notice that 
	 \begin{align*}
	\kappa(x_\infty)=\left(\psi'(x_\infty)\right)^{-1}\left((a-bx_\infty)\psi''(x_\infty)-(b+\rho)\psi'(x_\infty)\right)=-\frac{\sigma^2}{2}\frac{\psi'''(x_\infty)}{\psi'(x_\infty)}<0,
	 \end{align*}
	where we have used that $x_\infty$ solves $x_\infty-c=\frac{\psi'(x_\infty)}{\psi''(x_\infty)}$ for the first equality, and Lemma \ref{Properties}-(2) with $k=1$ for the second equality. Moreover, $$\kappa'(x)=\frac{bc-a}{(x-c)^2}<0,$$ since $a>bc$, which then yields $\kappa(x) < 0$ for all $x>x_\infty$. From the monotonicity and the negativity of $\kappa$, and the fact that $z_o\mapsto (x+\alpha z_o-c)$ is positive and increasing as $x\geq x_\infty>c$, one obtains that $z_o \mapsto (x+\alpha z_o-c)\kappa(x+\alpha z_o)$ is decreasing. Therefore, one has $U'(z_o)\leq 0$ for all $z_o > 0$ if $U'(0+) \leq 0$.
	
To prove that the right-derivative $U'(0+)$ is negative, we now explain how to employ in our setting the arguments of the proof of Lemma 6.7 in \cite{Becherer}. First of all, we discuss the standing Assumption 2.2 in \cite{Becherer}. Conditions C2 and C3 are satisfied for $f(x)\equiv x-c$. If $a-bc>0$, then Condition C5 in Assumption 2.2 of \cite{Becherer} is satisfied for $f(x)\equiv x-c$, $\hat{\sigma}\equiv \sigma$, $\beta+\delta \equiv \rho$, $\sigma\rho\hat{\sigma} \equiv a$, and $\beta\equiv b$. Moreover, all the other requirements in Assumption 2.2 of \cite{Becherer} are not needed in our case. Indeed, Condition C6 guarantees the existence and uniqueness of (in our terminology) $x_0$ and $x_{\infty}$, that we already have by Lemma \ref{UniqueSol}; Condition C4 only ensures a growth condition on the value function that we have from Proposition \ref{GrowthV}, whereas, in our setting, Condition C1 of \cite{Becherer} just means that the discount factor must be strictly positive.
	
Then, after reformulating our singular stochastic control problem as a calculus of variations problem where one seeks for a decreasing $C^1$ function triggering a strategy of reflecting type (see Section 4 in \cite{Becherer}), proceeding as in Section 5 of \cite{Becherer} (see in particular Theorem 5.6 therein), one can prove that our free boundary $F^{-1}$ is a (one-sided) local maximizer of our performance criterion \eqref{PC}. Hence, a contradiction argument as that in the proof of Lemma 6.7 in \cite{Becherer} also applies in our case and yields that $U'(0+) \leq 0$. This completes the proof.
\end{proof}

\begin{remark}
\label{rem:verificoBech}
\begin{enumerate} \hspace{10cm}
\item As we have seen, the proof of \eqref{cond6} in $\mathbb{S}_2$ when $a-bc>0$ requires a different analysis, and here we try to explain why a more direct approach seems not to lead to the desired result. Assuming $a-bc>0$, if one aims at proving \eqref{cond6} by studying the sign of $\mathcal{L}w-\rho w$ in $\mathbb{S}_2$, given that $z:=z(x,y)\geq 0$ for all $(x,y)\in\mathbb{S}_2$, one could try to prove that (cf.\ \eqref{widebarQ})
$$L(x,y):=a-bx-\rho(x-c) + \frac{1}{2}\rho\alpha z - b\alpha A(F(x-\alpha z))\psi'(x-\alpha z)$$
is negative for any $(x,y)\in \mathbb{S}_2$. Calculations, employing \eqref{eq1} and the definition of $A'$ (cf.\ \eqref{DA}), reveal that for any $y>0$ one has $L(F^{-1}(y),y)=\chi(F^{-1}(y))$, where, for any $u \in (x_{\infty},x_0]$, we have set
$$\chi(u):= (\rho + 2b)(\widehat{x}-u) + b\psi(u)\left[\frac{(u-c)\psi''(u) - \psi'(u)}{\psi''(u)\psi(u)-\psi'(u)^2} \right],$$
with $\widehat{x}:=\frac{a + (\rho + b)c}{\rho + 2b}<x_{\infty}$. By noticing that $A(F(x-\alpha z))\psi'(x-\alpha z)=w_x(x,y)-z$ in $\mathbb{S}_2$ (cf.\ \eqref{eq24}), one has that $L$ rewrites as $L(x,y)=a-bx-\rho(x-c) + \frac{1}{2}\rho\alpha z + b\alpha z - b\alpha w_x(x,y)$, and because $\alpha z_x < 1$ by \eqref{eq20} and $w_{xx} \geq 0$ by \eqref{eq25}, it is easy to see that $L_x<0$ on $\mathbb{S}_2$.

Hence, to prove that $L<0$ on $\mathbb{S}_2$ it would suffice to show that $\chi<0$ on $(x_{\infty},x_0]$. However, we have not been able to prove this property due to the unhandy implicit expression of the function $\psi$, even if a numerical investigation seems to confirm negativity of $\chi$. For this technical reason in \emph{Step 3-(ii)} of the proof of Theorem \ref{HJB} we have hinged on arguments as those originally developed in \cite{Becherer} to address the case $a-bc>0$.

\item It is also worth noticing that the calculus of variations approach of \cite{Becherer} would have not been directly applicable for any choice of the parameters. Indeed, when $a-bc<0$, the function $\kappa$ of \eqref{eq:kappa} is increasing and therefore has not the monotonicity required in Condition C5 of Assumption 2.2 of \cite{Becherer}. However, under such a parameters' restriction, direct calculations as those developed in \emph{Step 3-(i)} of the proof of Proposition \ref{HJB} lead to the desired result. This fact suggests that a combined use of the calculus of variations method and of the more standard direct study of the HJB equation could be successful in complex situations where neither of the two methods seem to leed to the proof of optimality of a candidate value function for any choice of the model's parameters.
\end{enumerate}
\end{remark}

We conclude by showing that $w$ of \eqref{CaFnc} identifies with the value function $V$. As a byproduct we also provide an optimal extraction rule. We first need the following technical result. Its proof follows by suitably adopting the classical result in \cite{dupuis}, upon considering the following joint process $(X,\zeta)$ as a (degenerate) diffusion in $\mathbb{R}^2$ with oblique reflection in the direction $(-\alpha,-1)$ at the $C^{\infty}$-free boundary $F$ (see also \cite{Becherer}, Remark 4.2).

\begin{lemma}
\label{Skoro}
Let $(x,y) \in \mathbb{R}\times (0,\infty)$, $F$ be given as in \eqref{boundary}, $z:=z(x,y)$ solving \eqref{optimalExercise}, and let $\Delta:=\Delta(x,y)=y\mathds{1}_{\{(x,y) \in \mathbb{S}_1\}}+z\mathds{1}_{\{(x,y) \in \mathbb{S}_2\}}$. Then there exists a (pathwise) unique $\mathbb{F}$-adapted continuous $(X,\zeta)$, with $\zeta$ increasing, such that
\begin{align*}
		X_t&\leq F^{-1}(y-\Delta-\zeta_t),\\
		dX_t&=\big(a-bX_t\big)dt+\sigma dW_t-\alpha d\zeta_t,\\
		d\zeta_t&=\mathds{1}_{\{X_t=F^{-1}(y-\Delta-\zeta_t)\}}d\zeta_t,
		\end{align*}
		for any $0 \leq t \leq \tau_{\zeta}$, $\tau_{\zeta}:=\inf\{t\geq 0:\, \zeta_t \geq y-\Delta\}$, and starting point $(X_0,\zeta_0)=(x-\alpha\Delta,0)$. 
\end{lemma}

 \begin{theorem}
\label{optsolOU}
	Recall the functions $F$ and $w$ from \eqref{boundary} and \eqref{CaFnc}, respectively. The function $w$ identifies with the value function $V$ from \eqref{ValueFnc}, and the optimal extraction strategy, denoted by $\xi^\star$, is given by
		\begin{align}
		\label{xistarOU}
		\xi^\star_t=
		\begin{cases}
		\Delta+\zeta_t,\quad &t\in[0,\tau_{\zeta}),\\
		y,\quad &t\geq\tau_{\zeta},
		\end{cases}
		\end{align} 
		with 	$\xi^\star_{0-}=0$, and with $\Delta$, $\zeta$, and $\tau_{\zeta}$ as in Lemma \ref{Skoro}.
\end{theorem}

\begin{proof}
We aim at applying Theorem \ref{VerificationTheorem}. We already know that $w\in C^{2,1}(\mathbb{R}\times [0,\infty))$ is a solution to the HJB equation \eqref{HJB1} by Lemma \ref{Smooth} and Proposition \ref{HJB}, and that satisfies $w(x,0)=0$ for all $x\in \mathbb{R}$. Moreover, the function $w$ is increasing with respect to $y$. To see that, notice that one has from \eqref{DA} that $A'(y)>0$, for $y>0$ (since the denominator of \eqref{DA} is positive by Lemma \ref{Properties}-(3) and the numerator is positive as well due to $F^{-1}(y) \geq x_{\infty}$), and this gives $w_y>0$ on $\mathbb{W}$ and on $\mathbb{S}_2$ (cf.\ \eqref{DerWait} and \eqref{eq26}). Also, one can easily check from \eqref{DerS1} that $w_y\geq 0$ on $\mathbb{S}_1$ because $y \leq (x-x_0)/\alpha$ and $x_0>c$.

To prove the upper bound in \eqref{eq15}, recall that (cf. \eqref{A(y)}) $$A(y)=\frac{(F^{-1}(y)-c)\psi'(F^{-1}(y))-\psi(F^{-1}(y))}{\alpha[\psi'(F^{-1}(y))^2-\psi''(F^{-1}(y))\psi(F^{-1}(y))]}, \quad y\geq 0.$$ 
Since $x_0\geq F^{-1}(y)\geq x_\infty$ for any $y \geq 0$, by using that $\psi,\,\psi'$ and $\psi''$ are continuous we have that there exists a constant $\overline{K}>0$ such that $A(y)\leq \overline{K}$ for all $y\geq0$. Hence, by \eqref{CaFnc} we have $w(x,y)\leq \overline{K}\psi(F^{-1}(y)) \leq \overline{K}\psi(x_0)$ for all $(x,y)\in\mathbb{W}$. Moreover, $0 \leq z(x,y)\leq y$ for all $(x,y)\in\mathbb{S}_2$ and thus $(x-c)z-\frac{1}{2}\alpha z\leq (x-c)z\leq(x-c)y$. Since the upper bound in \eqref{eq15} is clearly satisfied in $\mathbb{S}_1$, we conclude that there exists a constant $K>0$ such that $$w(x,y)\leq Ky(1+y)(1+|x|) \quad\text{for all }(x,y)\in\mathbb{R}\times (0,\infty).$$
	
As for the nonnegativity of $w$, notice that for all $(x,y)\in\mathbb{S}_1$ we have 
	\begin{align*}
	w(x,y)&=(x-c)y-\frac{1}{2}\alpha y^2\geq y\big[x-c-\frac{1}{2}(x-x_0)\big]\geq y\big[\frac{x_\infty-c}{2}+\frac{x_0-c}{2}\big]\geq 0,
	\end{align*}
	since $y\leq\frac{x-x_0}{\alpha}$, $x\geq F^{-1}(y)\geq x_\infty$ and $x_0>x_\infty>c$. Moreover, the nonnegativity of $\psi$ and $A$ imply $$w(x,y)\geq 0,\quad \text{for all }(x,y)\in\mathbb{W},$$ and also, given $(x,y)\in\mathbb{S}_2$, we have $$w(x,y)=A(F(x-\alpha z))\psi(x-\alpha z)+(x-c)z-\frac{1}{2}\alpha z^2\geq\int_{0}^{z}(x-\alpha u-c)du\geq\int_{0}^{z}(x_\infty-c)du\geq 0,$$ since $0 \leq z \leq \frac{x-x_\infty}{\alpha}$ and $x_{\infty}>c$. Therefore $w\geq0$ on $\mathbb{R} \times [0,\infty)$.
	
Now, since $\xi^\star$ satisfies \eqref{eq2VT} and \eqref{eq3VT}, by Theorem \ref{VerificationTheorem} we therefore conclude that $w$ identifies with $V$, and that $\xi^\star$ is an optimal extraction strategy.
\end{proof}

\begin{figure}
\centering
\includegraphics[width=1\textwidth]{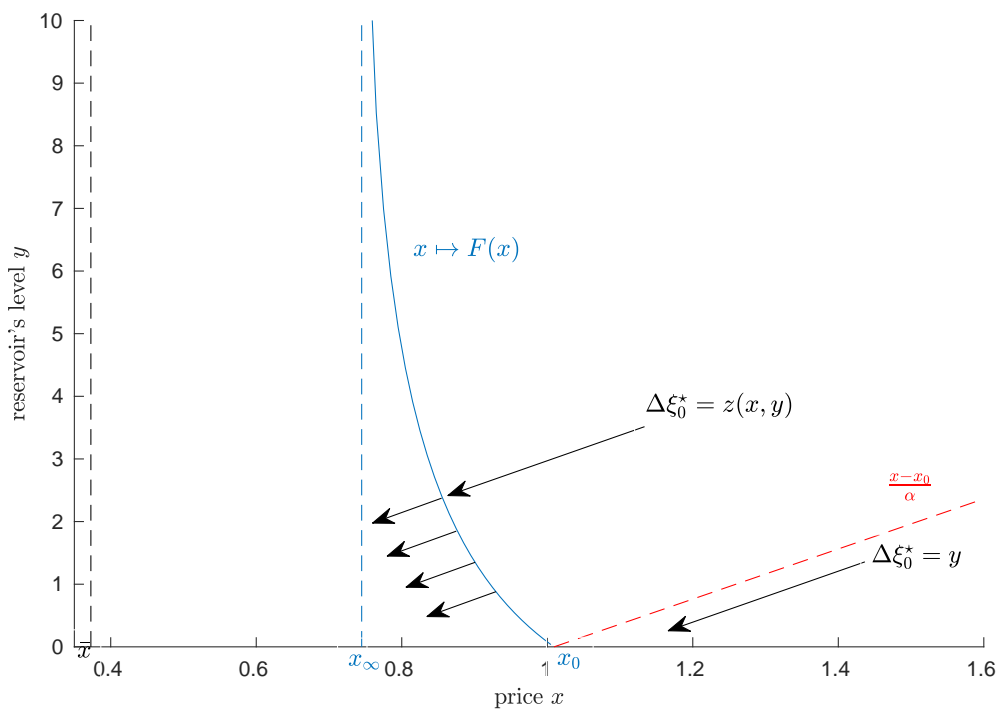}
\
\caption{A graphical illustration of the optimal extraction rule $\xi^{\star}$ (cf.\ \eqref{xistarOU}) and of the free boundary $F$. The plot has been obtained by using $a=0.4,\, \sigma=0.8,\, \rho=3/8,\, c=0.3,\, b=1,\, \alpha=0.25$, and by numerically evaluating the free boundary of \eqref{boundary}. The optimal extraction rule prescribes the following. In the region $\{(x,y) \in \mathbb{R} \times (0,\infty):\, y < F(x)\}$ it is optimal not to extract. If at initial time $(x,y)$ is such that $x > F^{-1}(y)$ and $y \leq (x-x_0)/\alpha$, then the reservoir should be immediately depleted. On the other hand, if $(x,y)$ is such that $x\geq F^{-1}(y)$ and $y > (x-x_0)/\alpha$, then one should make a lump sum extraction of suitable size $z(x,y)$, and then keep on extracting until the commodity is exhausted by just preventing the (optimally controlled) process $(X,Y)$ to leave the region $\{(x,y) \in \mathbb{R} \times (0,\infty):\, y \leq F(x)\}$.}
\end{figure}

%%%%%%%%%%%%%%%%%%%%%%%%%%%%%%%%%%%%%%%%%%%%%%%%%%%

\subsubsection{A Related Optimal Stopping Problem}
\label{sec:relatedOS}

In this section we show that the directional derivative $u:=\alpha V_x+V_y$ identifies with the value function of an optimal stopping problem. Such a result is consistent with that obtained - for a different model with Brownian dynamics - in \cite{karatzas2}, where connections between finite-fuel singular stochastic control problems and questions of optimal stopping have been studied. 

\begin{proposition}
\label{prop:relatedOS}
The function $u:\mathbb{R}\times[0,\infty)\mapsto\mathbb{R}$ defined by $$u(x,y):=\alpha V_x(x,y)+V_y(x,y)$$ admits the probabilistic representation
\begin{align}
\label{RepresentationStopp}
u(x,y)=\sup_{\tau\geq 0}\mathbb{E}\bigg[e^{-\rho\tau}\left(X^x_\tau-c\right)-\int_{0}^{\tau}e^{-\rho s}\alpha bA(y)\psi'(X^x_s)ds\bigg], \quad (x,y) \in \mathbb{R}\times[0,\infty),
\end{align}
where the optimization is taken over the set of $\mathbb{F}$-stopping times. Moreover, for $F$ as in \eqref{boundary}, we have that the stopping time
$$\tau^\star(x;y)=\inf\{t\geq 0: X_t^x \geq F^{-1}(y)\}, \quad (x,y) \in \mathbb{R}\times[0,\infty),$$
is optimal in \eqref{RepresentationStopp}.
\end{proposition}

\begin{proof}
For the rest of this proof, $y\in[0,\infty)$ will be given and fixed. Notice that $u(\cdot,y) \in C^1(\mathbb{R})$ by construction (cf.\ \eqref{eq1} and \eqref{eq2}). Moreover, direct calculations on \eqref{CaFnc} show that $u_{xx}(\cdot,y) \in L^{\infty}_{loc}(\mathbb{R})$. We now show that $u(\cdot,y)$ solves the HJB equation 
\begin{align}
\label{HJB2}
\max\Big\{\mathcal{L}w(x)-\rho w(x)-\alpha bA(y)\psi'(x),\, x-c-w(x)\Big\}=0, \quad \text{a.e.}\ x\in\mathbb{R}.
\end{align}

Recall the selling region $\mathbb{S}$ and the waiting region $\mathbb{W}$. Let $x\in\mathbb{R}$ be such that $(x,y)\in\mathbb{W}$, and notice that by \eqref{CaFnc} we have 
$$V_x(x,y)=A(y)\psi'(x),\quad \text{and}\quad V_y(x,y)=A'(y)\psi(x).$$
Then, since $u=\alpha V_x+V_y$,
\begin{align*}
&\cL u(x,y)-\rho u(x,y)-\alpha bA(y)\psi'(x)\\
=&\frac{1}{2}\sigma^2\left(\alpha A(y)\psi'''(x)+A'(y)\psi''(x)\right)+(a-bx)\left(\alpha A(y)\psi''(x)+A'(y)\psi'(x)\right)\\
&-(\rho+b)\alpha A(y)\psi'(x)-\rho A'(y)\psi(x) \\
=& \alpha A(y) \big(\cL \psi'(x)-(\rho+b) \psi'(x)\big) + A'(y) \big(\cL \psi(x)-\rho \psi(x)\big) =0,
\end{align*} 
upon using that $\psi^{(k)}$ satisfies Lemma \ref{Properties}-(2) with $k=0,1$.

Now, let $x\in\mathbb{R}$ be such that $(x,y)\in\mathbb{S}$, so that $u(x,y)=x-c$ (recall \eqref{cond2}). If $(x,y)\in\mathbb{S}_1$ then $x\geq x_0$, and using that $\alpha bA(y)\psi'(x)>0$ we obtain 
\begin{align*}
& \cL u(x,y)-\rho u(x,y)-\alpha bA(y)\psi'(x) =(a-bx)-\rho(x-c)-\alpha bA(y)\psi'(x)\\
& \leq a-(\rho+b)x+\rho c = (\rho + b)(\bar{x}-x) \leq 0,
\end{align*}
since $x_0 \geq \bar{x}$ by Lemma \ref{Propx_0xbar} in Appendix \ref{App:B}.

On the other hand, let $x\in\mathbb{R}$ be such that $(x,y)\in\mathbb{S}_2$, set $H(x,y):=\cL u(x,y) -\rho u(x,y)-\alpha bA(y)\psi'(x)$, and notice that $$\frac{\partial H(x,y)}{\partial x}=-(\rho+b)-\alpha b A(y)\psi''(x)<0,$$ due to the positivity of $A$ and $\psi''$. Thus, in order to prove that $\cL u(x,y) -\rho u(x,y)-\alpha bA(y)\psi'(x)\leq 0$ for all $(x,y)\in\mathbb{S}_2$, it is enough to prove that $H(F^{-1}(y),y)\leq 0$. Set $u:=F^{-1}(y)$; then, upon employing the definition of $A$ (cf. \eqref{A(y)}), we obtain
	\begin{align*}
	H(u,y)=&\left(\psi(u)\psi''(u)-\psi'(u)^2\right)^{-1}\times\\
	& \times \Big[(a-bu-\rho(u-c))\left(\psi(u)\psi''(u)-\psi'(u)^2\right)+b(u-c)\psi'(u)^2-b\psi(u)\psi'(u)\Big]\\
	=&\frac{\sigma^2}{2}\left(\psi(u)\psi''(u)-\psi'(u)^2\right)^{-1}\times\\
	& \times \Big[\psi'''(u)\left[(u-c)\psi'(u)-\psi(u)\right]-\psi''(u)\left[(u-c)\psi''(u)-\psi'(u)\right]\Big] < 0,
	\end{align*} 
where we have applied Lemma \ref{Properties}-(2) with $k=0$ and $k=1$ for the last equality, and the last inequality follows from Corollary \ref{Order} since $x_{\infty}<u\leq x_0$. Hence, $\cL u(x,y) -\rho u(x,y)-\alpha bA(y)\psi'(x) \leq 0$ on $\mathbb{S}_2$.

Finally, from Proposition \ref{HJB} we have $x-c-u(x,y)\leq 0$ for any $x \in \mathbb{R}$.

The previous inequalities show that $u(\cdot,y)$ identifies with a $W^{2,\infty}_{loc}(\mathbb{R})$-solution to \eqref{HJB2}. Then, a standard verification theorem based on an application of (a generalized version of) It\^o's formula, implies that $u(\cdot,y)$ admits representation \eqref{RepresentationStopp} and that the stopping time $\tau^\star(x;y)=\inf\{t\geq 0: X_t^x \geq F^{-1}(y)\}$ attains the supremum.
\end{proof}

\begin{remark}
\label{rem:OS}
A few comments are worth being done.
\begin{itemize}
\item[1.] With regard to the connection between problems of singular stochastic control and questions of optimal stopping (see, e.g., \cite{EKK}, \cite{EKK2}, \cite{KS84}, and \cite{karatzas2} as early contributions, and the introduction of the recent \cite{DeAFe16} for a richer literature review), we can interpret the stopping time $\tau^\star(x;y)$ as the optimal time at which an additional unit of the commodity should be extracted. Indeed, the underlying process at that time is such that, in economic terms, equality between the marginal expected optimal profit (i.e.\ $\alpha V_x + V_y$) and the marginal instantaneous net profit from extraction (i.e.\ $x-c$) holds.

\item[2.] If we do not consider price impact in our model (i.e.\ we take $\alpha=0$), it can be easily seen that the value function of the resulting optimal extraction problem $V$ is such that
$$V_y(x,y)= \sup_{\tau\geq 0}\mathbb{E}\Big[e^{-\rho\tau}\left(X^x_\tau-c\right)\Big],$$
a result that is clearly consistent with \eqref{RepresentationStopp}. The integral term 
$$-\int_{0}^{\tau}e^{-\rho s} \alpha b A(y)\psi'(X^x_s)ds$$ 
appearing in \eqref{RepresentationStopp} can then be seen as a running cost/penalty whose effect increases with increasing price impact $\alpha$. 

\item[3.] It can be checked that the arguments of the proof of Proposition \ref{prop:relatedOS} carry over also to the case of a fundamental price given by a drifted Brownian motion, i.e.\ when $b=0$ (cf.\ Section \ref{sec:ABM}). As one would expect by setting $b=0$ in the right-hand side of \eqref{RepresentationStopp}, in such a case it holds
$$\alpha V_x(x,y) + V_y(x,y) = \sup_{\tau\geq 0}\mathbb{E}\Big[e^{-\rho\tau}\left(X^x_\tau-c\right)\Big],$$
so that the stopping problem related to the optimal extraction problem does not depend on the current level of the reservoir $y$. This explains why, in in the drifted Brownian motion case studied in Section \ref{sec:ABM}, the free boundary $x^{\star}$ triggering the optimal extraction rule is $y$-independent.
\end{itemize}
\end{remark}

\section{Comparative Statics Analysis}
\label{sec:compstatics}

In this section, we study the sensitivity of the solution to the extraction problem separately for the case of a fundamental price given by a drifted Brownian motion (Section \ref{comp:ABM}) and by an Ornstein-Uhlenbeck process (Section \ref{comp:OU}). In particular, in Section \ref{comp:ABM} we analytically determine the dependency of the free boundary $x^{\star}$ of \eqref{xstarABM} and of the value function \eqref{CaValueABM} on the parameters $a$ and $\sigma$. In Section \ref{comp:OU} we study analytically how the value function \eqref{CaFnc} and the critical price levels $x_0$ and $x_\infty$ from Lemma \ref{xinfty} depend on $a$ and $\sigma$, and, numerically, the sensitivity of the free boundary $F$ with respect to $a$, $\sigma$ and $b$.

\subsection{Sensitivity Analysis in the Case of a Drifted Brownian Motion Fundamental Price}
\label{comp:ABM}

Here we assume $b=0$ in \eqref{affectedX}. Thanks to the explicit formula \eqref{xstarABM}, studying the sensitivity of the free boundary $x^{\star}$ with respect to the parameters $a$ and $\sigma$ is a simple exercise of differentiation.

\begin{proposition}
\label{propx0}
The free boundary $x^\star$ of \eqref{xstarABM} is increasing with respect to both $a$ and $\sigma$. 
\end{proposition}

\begin{proof}
We look at the parameter $n$ of \eqref{eq47} as a function of $a$ and $\sigma$; that is, we set
$$n(a,\sigma):=-\frac{a}{\sigma^2}+\sqrt{\Big(\frac{a}{\sigma^2}\Big)^2+2\frac{\rho}{\sigma^2}}.$$ 
Then, it is not hard to find by direct calculations that
\begin{equation}
\label{eq67}
n_a(a,\sigma) =\frac{1}{\sigma^4}\left(\frac{a}{\sqrt{\Big(\frac{a}{\sigma^2}\Big)^2+2\frac{\rho}{\sigma^2}}}-\sigma^2\right),
\end{equation}
and
\begin{equation}
\label{eq68}
n_{\sigma}(a,\sigma) =\frac{2}{\sigma^3}\left(a-\frac{\frac{a^2}{\sigma^2}+\rho}{{\sqrt{\Big(\frac{a}{\sigma^2}\Big)^2+2\frac{\rho}{\sigma^2}}}}\right).
\end{equation} 

Clearly, if $a\leq 0$ one has $n_a \leq 0$ and $n_{\sigma} \leq 0$. Then, suppose $a>0$ and notice that
\begin{equation}
\label{ineq-n}
\sqrt{\Big(\frac{a}{\sigma^2}\Big)^2+2\frac{\rho}{\sigma^2}}\geq\frac{a}{\sigma^2} \quad \text{and}\quad \sqrt{\Big(\frac{a}{\sigma^2}\Big)^2+2\frac{\rho}{\sigma^2}}\leq\frac{a}{\sigma^2}+\frac{\rho}{a},
\end{equation}
where the second inequality above follows by an application of the binomial formula. By using the first inequality of \eqref{ineq-n} in \eqref{eq67}, and the second inequality of \eqref{ineq-n} in \eqref{eq68}, one easily finds that $n_{a}(a,\sigma) \leq 0$, as well as $n_{\sigma}(a,\sigma) \leq 0$.

Finally, the claim follows since $x^\star$ is decreasing with respect to $n$ (cf.\ \eqref{eq47}). 
\end{proof}

\begin{proposition}
\label{propV}
The value function $V$ defined in \eqref{ValueFnc} is increasing with respect to $a$ and $\sigma$. 
\end{proposition}

\begin{proof}
Let $\hat{a}>a$ and $\hat{\sigma}>\sigma$. We show the monotonicity with respect to $a$ and $\sigma$ separately in two steps.
\vspace{0.25cm}

\emph{Step 1.} Let $(x,y)\in \mathbb{R}\times (0,\infty)$ be given and fixed. For any $\xi\in\mathcal{A}(y)$, we denote by $\widehat{X}_t^{x,\xi}$ the solution to \eqref{affectedX} when $b=0$ and the drift is $\hat{a}$. One clearly has $\widehat{X}_t^{x,\xi}\geq X_t^{x,\xi}$ $\mathbb{P}$-a.s.\ for any $t\geq 0$. Therefore $\widehat{\mathcal{J}}(x,y,\xi)\geq\mathcal{J}(x,y,\xi)$ for any $\xi\in\mathcal{A}(y)$, where $\hat{\mathcal{J}}$ is given by \eqref{PC} with underlying state $(\widehat{X}^{x,\xi},Y^{y,\xi})$. Hence, we conclude 
	\begin{align*}
	\widehat{V}(x,y)\geq V(x,y),\quad \forall(x,y)\in\mathbb{R}\times [0,\infty),
	\end{align*}
	where $\widehat{V}(x,y):=\sup_{\xi \in \mathcal{A}(y)}\hat{\mathcal{J}}(x,y,\xi)$.
\vspace{0.25cm}
	
\emph{Step 2.} To prove the monotonicity of $V$ with respect to $\sigma$ we adapt to our setting ideas from Theorem 4 in \cite{Alvarez2000}. Let $\widehat{V}$ be the value function when the volatility coefficient in \eqref{affectedX} is $\hat{\sigma}$. Recall $\mathcal{L}$ as in \eqref{InfinitesOp}, and let $\widehat{\mathcal{L}}$ be as in \eqref{InfinitesOp} but with volatility coefficient $\widehat{\sigma}$. Then, for all $(x,y)\in\mathbb{R}\times (0,\infty)$ we have

\begin{align}
\label{cond10}
\begin{split}
\mathcal{L}\widehat{V}(x,y)-\rho\widehat{V}(x,y)&=\frac{\hat{\sigma}^2}{2}\widehat{V}_{xx}(x,y)+a\widehat{V}_x(x,y)-\rho\widehat{V}(x,y)+\frac{(\sigma^2-\hat{\sigma}^2)}{2}\widehat{V}_{xx}(x,y)\\&=\widehat{\mathcal{L}}\widehat{V}(x,y)+\frac{(\sigma^2-\hat{\sigma}^2)}{2}\widehat{V}_{xx}(x,y) \leq \frac{(\sigma^2-\hat{\sigma}^2)}{2}\widehat{V}_{xx}(x,y)\leq 0,
\end{split}
\end{align}
since $\widehat{V}(\cdot,y)$ is convex by the second equations in \eqref{eq70} and \eqref{eq60}, and the second equation of \eqref{eq72}. Furthermore, since $\widehat{V}$ is the value function of the optimal extraction problem when in \eqref{affectedX} the volatility is $\hat{\sigma}$, $\widehat{V}$ must satisfy
	\begin{align}\label{cond11}
	-\alpha\widehat{V}_x(x,y)-\widehat{V}_y(x,y)+(x-c)\leq 0,
	\end{align}
for all $(x,y)\in \mathbb{R}\times (0,\infty)$, and $\widehat{V}(x,0)=0$ for all $x \in \mathbb{R}$. Now, arguing as in the first step of the proof of Theorem \ref{VerificationTheorem}, by using \eqref{cond10} and \eqref{cond11}, we obtain $\widehat{V} \geq V$, and thus the claimed monotonicity.
\end{proof}

Propositions \ref{propx0} and \ref{propV} show that the higher the level of the drift $a$ is, and hence the higher the expected prices are, the later the company starts extracting in order to obtain larger profits. Moreover, higher uncertainty, and hence larger price's fluctuations, are exploited by the company that then sells the commodity at higher prices and increases the resulting profits.

%%%%%%%%%%%%%%%%%%%%%%%%%%%%%%%%%%%%%%%%%%%%%%%%%%%%%

\subsection{Sensitivity Analysis in the Case of an Ornstein-Uhlenbeck Fundamental Price}
\label{comp:OU}

We start by studying the sensitivity of $x_0$ and $x_\infty$ (cf.\ Lemma \ref{xinfty}) on the model parameters $a$ and $\sigma$. In the following, when needed, we write $g(\cdot;a,\sigma)$ in order to emphasize the dependency of a given real-valued function $g$ with respect to $a$ and $\sigma$. 

Recall that the fundamental increasing solution to the equation $(\cL - \rho)u=0$ is given by \eqref{eq34} (see also \eqref{eq35}). In the following, when needed, we denote by $\psi^{(k)}(x;a,\sigma)$ the $k-$th derivative with respect to $x$ of $\psi$. By an application of the dominated convergence theorem one obtains the relation
\begin{align}\label{eq8}
\frac{\partial \psi^{(k)}}{\partial a}(x;a,\sigma):=\psi^{(k)}_a(x;a,\sigma)&=-\frac{1}{b}\psi^{(k+1)}(x;a,\sigma),\quad\text{for all }k\in\mathbb{N} \cup \{0\}.
\end{align}
Analogously, one finds
\begin{align}\label{eq10}
\frac{\partial \psi^{(k)}}{\partial \sigma}(x;a,\sigma):=\psi^{(k)}_{\sigma}(x;a,\sigma)&=\bigg(\frac{a-bx}{b\sigma}\bigg)\psi^{(k+1)}(x;a,\sigma)-\frac{k}{\sigma}\psi^{(k)}(x;a,\sigma),
\end{align}
for all $k\in\mathbb{N} \cup \{0\}$

By employing \eqref{eq8}, and Lemma \ref{Properties}, one can easily prove the next result.
\begin{lemma}
\label{lem:a}
One has that
\begin{align}
\label{eqpartialpsia}
\frac{\partial(\psi^{(k)}(x;a,\sigma)/\psi^{(k+1)}(x;a,\sigma))}{\partial a}=\frac{\psi^{(k)}(x;a,\sigma)\psi^{(k+2)}(x;a,\sigma)-\psi^{(k+1)}(x;a,\sigma)^2}{b\psi^{(k+1)}(x;a,\sigma)^2}>0,
\end{align}
\end{lemma}

The proof of the next result can be found in Appendix \ref{App:A}. It employs \eqref{eq10}.
\begin{lemma}\label{sigma}
	One has that
	\begin{align}
	\label{eq38}	
	&\frac{\partial(\psi^{(k)}(x;a,\sigma)/\psi^{(k+1)}(x;a,\sigma))}{\partial\sigma} \\
	& =\frac{(a-bx)[\psi^{(k+1)}(x;a,\sigma)^2-\psi^{(k)}(x;a,\sigma)\psi^{(k+2)}(x;a,\sigma)]+b\psi^{(k+1)}(x;a,\sigma)\psi^{(k)}(x;a,\sigma)}{b\sigma\psi^{(k+1)}(x;a,\sigma)^2}>0. \nonumber	
	\end{align}
\end{lemma}

The previous results on the dependency of $\psi/\psi_x$ with respect to $a$ and $\sigma$ (i.e.\ \eqref{eqpartialpsia} and \eqref{eq38}) allow us to determine the dependency of $x_0$ and $x_\infty$ on $a$ and $\sigma$ as well. One may intuitively expect that the company exploits a higher mean reversion level, and thus sells the commodity at higher prices. As an indication of this, we indeed find that $x_0$, $x_\infty$, and the value function $V$ increase as $a$ increases.

In the following we denote by $x_0$, $x_{\infty}$ the unique solutions on $(c,\infty)$ to $(x-c)\psi_x(x;a,\sigma)-\psi(x;a,\sigma)=0$ and $(x-c)\psi_{xx}(x;a,\sigma)-\psi_x(x;a,\sigma)=0$, respectively. Also, $V(x,y)$ denotes the value function when in \eqref{affectedX} the mean-reversion level is $a/b$ and the volatility is $\sigma$.

\begin{proposition}
	Let $\hat{a}>a$, and denote by $\hat{x}_0$ and $\hat{x}_\infty$ the unique solutions on $(c,\infty)$ to $(x-c)\psi_x(x;\hat{a},\sigma)-\psi(x;\hat{a},\sigma)=0$ and $(x-c)\psi_{xx}(x;\hat{a},\sigma)-\psi_x(x;\hat{a},\sigma)=0$, respectively. Furthermore, we denote by $\widehat{V}(x,y)$, $(x,y)\in \mathbb{R}\times[0,\infty)$, the value function when in \eqref{affectedX} the mean-reversion level is $\hat{a}/b$ and the volatility is $\sigma$.
	We have 
	\begin{align*}
	\hat{x}_0>x_0\quad \text{and} \quad \hat{x}_\infty>x_\infty,
	\end{align*}
and
	\begin{align}\label{staement1}
	\widehat{V}(x,y)\geq V(x,y),\quad \forall(x,y)\in\mathbb{R}\times [0,\infty).
	\end{align}
\end{proposition}

\begin{proof}
For any given $q \in \mathbb{R}$ and $\sigma>0$, set $H(x;q,\sigma):=(x-c)\psi_x(x;q,\sigma)-\psi(x;q,\sigma)$, $x\in \mathbb{R}$. We have $H_x(x;q,\sigma)>0$ for all $x>c$. Moreover,
 \begin{align*}
 H(\hat{x}_0;a,\sigma)=\frac{{\psi}(\hat{x}_0;\hat{a},\sigma)}{{\psi}_x(\hat{x}_0;\hat{a},\sigma)}\psi_x(\hat{x}_0;a,\sigma)-\psi(\hat{x}_0;a,\sigma)>0 = H(x_0;a,\sigma),
 \end{align*}
where we have used that $H(\hat{x}_0;\hat{a},\sigma) =0$ for the first equality, and Lemma \ref{lem:a} with $k=0$ for the inequality. Thus, by monotonicity of $H(\cdot;q,\sigma)$ on $(c,\infty)$, we have $\hat{x}_0>x_0$. Analogously, we can prove that $\hat{x}_\infty>x_\infty$ by employing Lemma \ref{lem:a} with $k=1$.

In order to prove \eqref{staement1}, we can proceed in the same way as in \emph{Step 1} of the proof of Proposition \ref{propV}.
\end{proof}

The next proposition shows that the critical price levels $x_0$ and $x_\infty$ increase as the price's fluctuations become larger.
\begin{proposition}
\label{CSsigma}
	Let $\hat{\sigma}>\sigma$, and denote by $\hat{x}_0$ and $\hat{x}_\infty$ the unique solutions on $(c,\infty)$ to $(x-c)\psi_x(x;a,\hat{\sigma})-\psi(x;a,\hat{\sigma})=0$ and $(x-c)\psi_{xx}(x;a,\hat{\sigma})-\psi_x(x;a,\hat{\sigma})=0$, respectively. Furthermore, denote by $\widehat{V}$ the value function when in \eqref{affectedX} the mean-reversion level is $a/b$ and the volatility is $\hat{\sigma}$. We have 
	\begin{align*}
	\hat{x}_0>x_0\quad \text{ and} \quad \hat{x}_\infty>x_\infty,
	\end{align*}
	and 
	\begin{align}\label{statement2}
	\widehat{V}(x,y)\geq V(x,y),\quad \forall(x,y)\in\mathbb{R}\times\mathbb{R}_+.
	\end{align}
\end{proposition}

\begin{proof}
For any given $q >0$ and $a \in \mathbb{R}$, set $H(x;a,q):=(x-c)\psi_x(x;a,q)-\psi(x;a,q)$, $x\in \mathbb{R}$. We have $H_x(x;a,q)>0$ for all $x>c$. Moreover, using that $H(\hat{x}_0;a,\hat{\sigma})=0$ we have 
 \begin{align*}
 H(\hat{x}_0;a,\sigma)=\frac{{\psi}(\hat{x}_0;a,\hat{\sigma})}{{\psi}_x(\hat{x}_0;a,\hat{\sigma})}\psi_x(\hat{x}_0;a,\sigma)-\psi(\hat{x}_0;a,\sigma)>0 = H(x_0;a,\sigma),
 \end{align*}
where the inequality is due to Lemma \ref{sigma} with $k=0$. Since $H(\cdot;a,q)$ is increasing for all $x>c$ we have $\hat{x}_0>x_0$. Analogously, we can prove that $\hat{x}_\infty>x_\infty$ by Lemma \ref{sigma} with $k=1$.

To prove \eqref{statement2} we can use the arguments employed in \emph{Step 2} of the proof of Proposition \ref{propV}, upon noticing that $\widehat{V}(\cdot,y)$ is convex by the second equations in \eqref{DerWait} and \eqref{DerS1}, and \eqref{eq25} (recall that $A$ is positive and $\psi$ is convex).
\end{proof}

The semi-explicit nature of our results allows us to easily study numerically the dependency of the free boundary $F$ with respect to $a$. This is shown in Figure \ref{fig3}. We see that $F$ increases as $a$ increases: the higher the level of mean reversion is, the later the company starts extracting in order to obtain larger profits.

\begin{figure}
	\centering
	\includegraphics[width=0.75\textwidth]{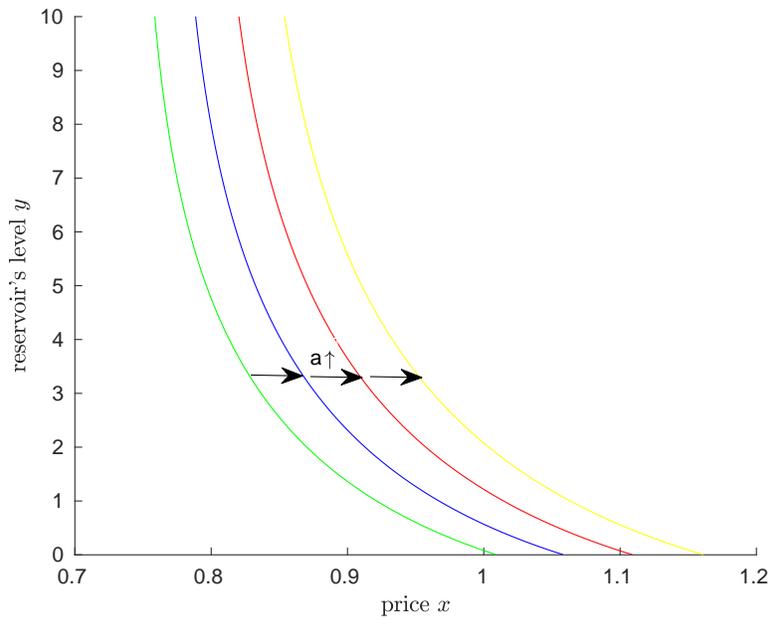}
	\caption{A drawing of the free boundary $x \mapsto F(x)$ for $b=1,\, \sigma=0.8,\, \rho=3/8,\, c=0.3,\, \alpha=0.25$ and various values for $a$: $a=0.4$ (green), $a=0.5$ (blue), $a=0.6$ (red), and $a=0.7$ (yellow).}
	\label{fig3}
\end{figure}

Figure \ref{fig2} shows the dependency of the curve $x\mapsto F(x)$ with respect to $\sigma$. We see that the whole curve $F$ increases as $\sigma$ increases. We thus conclude that higher uncertainty, and hence higher fluctuations around the mean-reversion level, are exploited by the company which then sells the commodity at higher prices and increases its profits.

\begin{figure}
\centering
\includegraphics[width=0.75\textwidth]{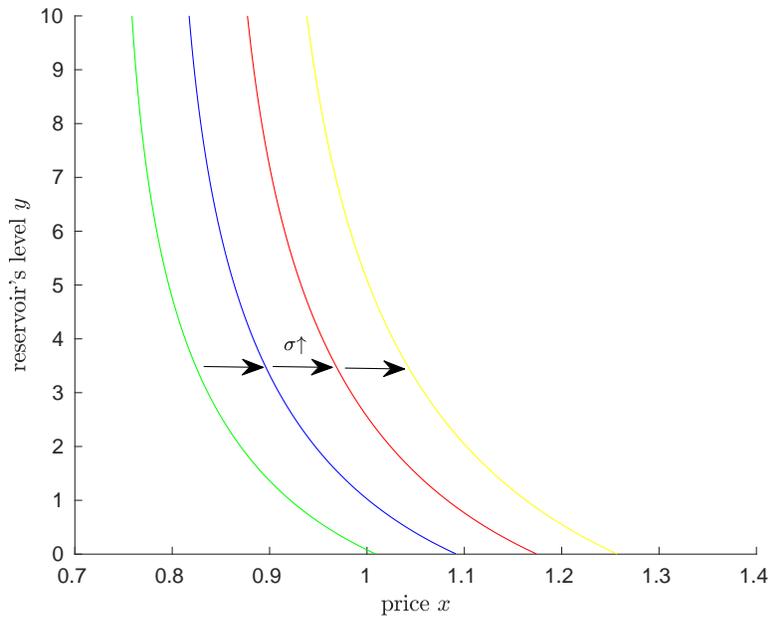}
\caption{A drawing of the free boundary $x \mapsto F(x)$ for $a=0.4,\, b=1,\, \rho=3/8,\, c=0.3,\, \alpha=0.25$ and various values for the volatility: $\sigma=0.8$ (green), $\sigma=0.9$ (blue), $\sigma=1$ (red), and $\sigma=1.1$ (yellow).}
\label{fig2}
\end{figure}

In Figure \ref{fig4}, we can observe the sensitivity of the free boundary $F$ with respect to $b$. Differently to what it is happening when increasing $\sigma$ and $a$, now the whole curve $F$ increases as $b$ decreases, and in fact, as $b \downarrow 0$, it converges to $x^\star$, which is the free boundary in the case $b=0$ (i.e.\ related to the drifted Brownian motion case). This fact might be interpreted by saying that, if $a>0$, a lower value of $b$ leads the company to wait more since it expects to be able to sell the commodity at higher prices in the future.

\begin{figure}
	\centering
	\includegraphics[width=0.75\textwidth]{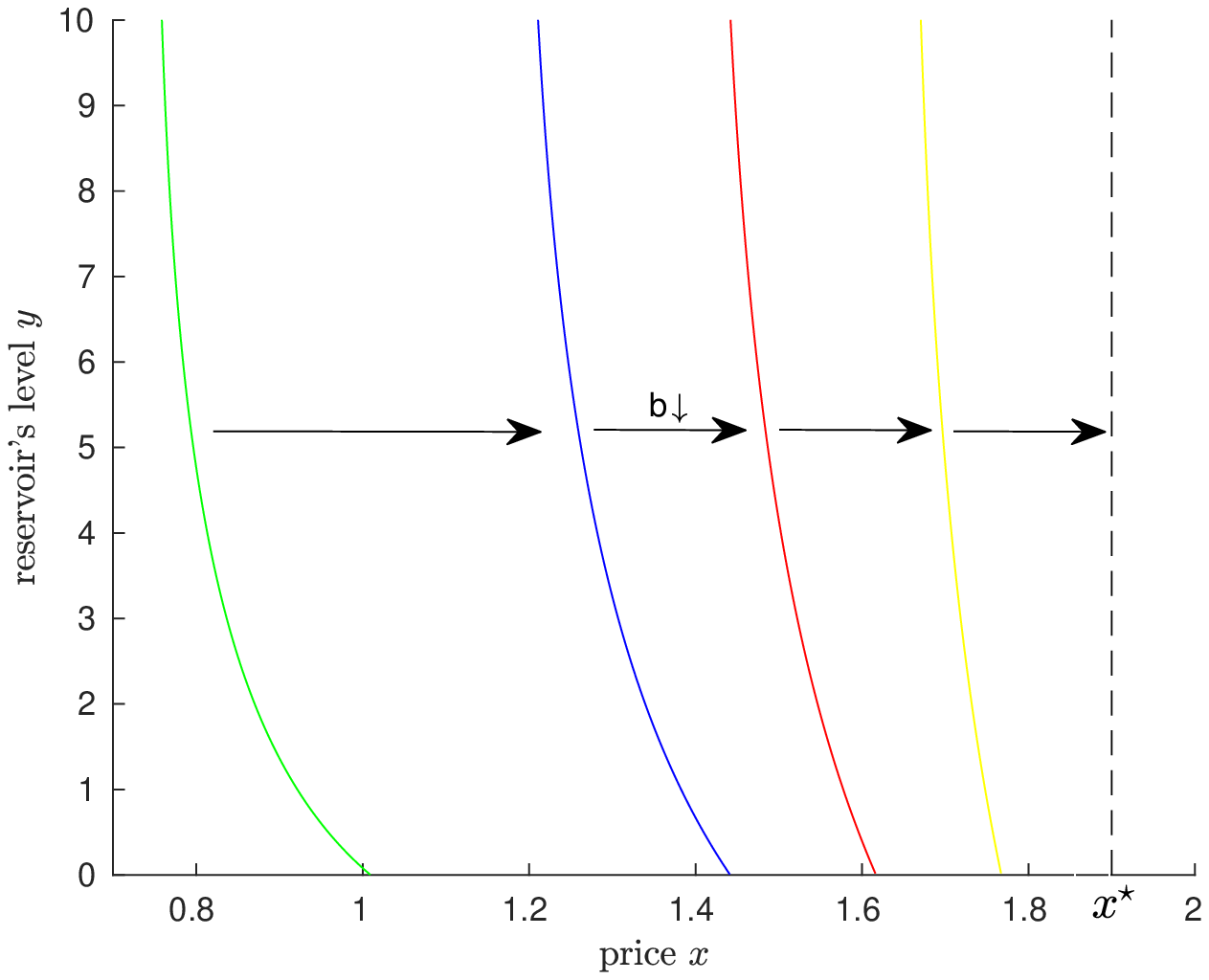}
	\caption{A drawing of the free boundary $x \mapsto F(x)$ for $a=0.4,\, \sigma=0.8,\, \rho=3/8,\, c=0.3,\, \alpha=0.25$ and various values for the mean reversion speed: $b=1$ (green), $b=0.25$ (blue), $b=0.125$ (red), and $b=0.05$ (yellow).}
	\label{fig4}
\end{figure}

%%%%%%%%%%%%%%%%%%%%%%%%%%%%%%%%%%%%%%%%%%%%%%%%%%%%%%%%%%%%%%%%%%%%%%%%%%%%%%%%%%%%%%%%%%%%%%%%%%%%%%%%%%%%%%%%%%%%%%%%%%%%%%%%%%%%%%%%%%%%%%%

\appendix
\section{Proofs of Results from Sections \ref{sec:OU} and \ref{comp:OU}}	
\label{App:A}

\renewcommand{\theequation}{A-\arabic{equation}}

\emph{Proof of Lemma \ref{Properties}.}

\begin{itemize}
\item [(1)] We refer the reader to \cite{jeanblanc}, among others. Moreover, the strict convexity of $\psi$ can be checked by direct calculations on \eqref{eq34}.

\item [(2)] 
Define the function $f:\mathbb{R}_+\times\mathbb{R}\to\mathbb{R}_+$ by 
$$f(t,x)=\frac{1}{\Gamma(\frac{\rho}{b})} t^{\big(\frac{\rho}{b}-1\big)}e^{-\frac{t^2}{2}+t\big(\frac{bx-a}{\sigma b}\big)\sqrt{2b}},$$ 
that, once differentiated with respect to $x$, yields
$$f_x(t,x)=\frac{\rho\sqrt{2b}}{b\sigma}\frac{1}{\Gamma(\frac{\rho+b}{b})} t^{\big(\frac{\rho+b}{b}-1\big)}e^{-\frac{t^2}{2}+t\big(\frac{bx-a}{\sigma b}\big)\sqrt{2b}}.$$

Notice that $f$ is the integrand appearing in \eqref{eq35} for $\beta=-\frac{\rho}{b}$. Then, differentiating \eqref{eq34} with respect to $x$, and invoking the dominated convergence theorem, we obtain
$$\psi'(x)\, \propto \,  e^{\frac{(bx-a)^2}{2\sigma^2b}}D_{-\frac{\rho+b}{b}}\bigg(-\frac{bx-a}{\sigma b}\sqrt{2b}\bigg),$$
upon noticing that $f_x(t,x)$ is the integrand of $D_{-\frac{\rho+b}{b}}\bigg(-\frac{bx-a}{\sigma b}\sqrt{2b}\bigg)$ (cf.\ \eqref{eq35}).

Hence, $\psi'$ can be identified (modulo a constant) as the positive strictly increasing fundamental solution to $(\cL-(\rho+b))u=0$, and by direct calculations it can be checked that it is strictly convex. By iterating the previous argument, we see that, for any $k\in\mathbb{N}$, the function $\psi^{(k)}$ is strictly convex and identifies with the positive strictly increasing fundamental solution to $(\cL-(\rho+kb))u=0$.

\item [(3)] We define the function $f^{(k)}:\mathbb{R}_+\times\mathbb{R}\to\mathbb{R}_+$ by $$f^{(k)}(t,x)=\frac{\big(\sqrt{{2b}}/{\sigma}\big)^\frac{k}{2}}{\Gamma(\frac{\rho}{b})^{\frac{1}{2}}}t^{\frac{1}{2}\big(\frac{\rho}{b}+k-1\big)}e^{-\frac{t^2}{4}+\frac{t}{2}\big(\frac{bx-a}{\sigma b}\big)\sqrt{2b}}.$$
By direct calculations, we find
\begin{align*}
\psi^{(k+1)}(x)=\int_{0}^{\infty}f^{(k+2)}(t,x)f^{(k)}(t,x)dt,\quad x\in\mathbb{R},
\end{align*}
that, by the help of H\"older's inequality (which is strict as $f^{(k)}(\cdot,x)$ is not a multiple of $f^{(k+2)}(\cdot,x)$), gives
\begin{align*}
\Bigg(\int_{0}^{\infty}f^{(k+2)}(t,x)f^{(k)}(t,x)dt\Bigg)^2<\int_{0}^{\infty}\big(f^{(k+2)}(t,x)\big)^2 dt\int_{0}^{\infty}\big(f^{(k)}(t,x)\big)^2dt.
\end{align*}
The latter is in fact equivalent to $$\psi^{(k+2)}(x)\psi^{(k)}(x)-\psi^{(k+1)}(x)^2>0.$$ 
\end{itemize}
\vspace{0.35cm}

\emph{Proof of Lemma \ref{UniqueSol}.}
\vspace{0.25cm}

Let $k\in\mathbb{N} \cup \{0\}$ be given and fixed, and define $\Lambda(x):=(x-c){\psi^{(k+1)}(x)}-\psi^{(k)}(x)$, $x\in\mathbb{R}$. We then have the following.
	\begin{itemize}
		\item [(i)] For $x\leq c$, it is readily seen that $\Lambda(x)<0$.
		
		\item [(ii)] One has $\Lambda(x)>0$ for all $x>c+\frac{\psi(c)}{\psi^{\prime}(c)}$. To see this, rewrite $\Lambda(x)=\psi^{(k)}(x)\big[(x-c)\frac{\psi^\prime(x)}{\psi(x)}-1\big]$, and notice that by Lemma \ref{Properties} 
		$$\Big(\frac{\psi^{\prime}(x)}{\psi(x)}\Big)'=\frac{\psi^{\prime\prime}(x)\psi(x)-(\psi^{\prime}(x))^2}{(\psi(x))^2}>0.$$ 
		Hence, for all $x>c+\frac{\psi(c)}{\psi^{\prime}(c)}>c$ one has that $\frac{\psi^{\prime}(x)}{\psi(x)}>\frac{\psi^{\prime}(c)}{\psi(c)}$, which implies 
		$$(x-c)\frac{\psi^{\prime}(x)}{\psi(x)}-1>(x-c)\frac{\psi^{\prime}(c)}{\psi(c)}-1>0,$$
		for all $x>c+\frac{\psi(c)}{\psi^{\prime}(c)}$. The latter clearly gives $\Lambda(x)>0$ for all $x>c+\frac{\psi(c)}{\psi^{\prime}(c)}$. 
		\end{itemize}
		
	Since $\Lambda'(x)=(x-c)\psi^{(k+2)}(x)>0$ for all $x>c$, we conclude from (i) and (ii) that there exists a unique solution on $(c,\infty)$ to the equation $\Lambda(x)=0$ by continuity of $\Lambda$.
\vspace{0.35cm}

\emph{Proof of Lemma \ref{xinfty}.}
\vspace{0.25cm}

	We argue by contradiction, and we suppose $x_{\infty}\geq x_0$. Then by definition of $x_0$ and $x_\infty$ we have
	\begin{align}\label{eq3}
	x_0-x_\infty=(x_0-c)-(x_\infty-c)=\frac{\psi(x_0)}{\psi^\prime(x_0)}-\frac{\psi^\prime(x_\infty)}{\psi^{\prime\prime}(x_\infty)}.
	\end{align}
	Since by Lemma \ref{Properties} 
	$$\Big(\frac{\psi(x)}{\psi^\prime(x)}\Big)^\prime=\frac{\psi^\prime(x)^2-\psi(x)\psi^{\prime\prime}(x)}{\psi^\prime(x)^2}<0,\quad\text{for any $x\in\mathbb{R}$,}$$ 
	we have by \eqref{eq3} that
	$$x_0-x_\infty\geq\frac{\psi(x_\infty)}{\psi^\prime(x_\infty)}-\frac{\psi^\prime(x_\infty)}{\psi^{\prime\prime}(x_\infty)}>0,$$
	again due to Lemma \ref{Properties}. But this contradicts $x_{\infty}\geq x_0$.
\vspace{0.35cm}

\emph{Proof of Lemma \ref{Existencez}.}
\vspace{0.25cm}

First of all notice that for the existence of a solution $z$ to \eqref{optimalExercise} it is necessary that $y-z\geq 0$ since $F\geq 0$, and that $x-\alpha z\in(x_\infty,x_0]$ since the domain of $F$ is $(x_\infty,x_0]$. Hence, if a solution to \eqref{optimalExercise} exists, it must be such that $z(x,y) \in (\frac{x-x_0}{\alpha},\frac{x-x_\infty}{\alpha} \wedge y]$, for all $(x,y)\in\mathbb{S}_2$.

Let $(x,y)\in\mathbb{S}_2$ with $y>F(x)$ be given and fixed, and define $R(z)=y-z-F(x-\alpha z)$, for $z\in(\frac{x-x_0}{\alpha},\frac{x-x_\infty}{\alpha} \wedge y)$. Then, one has $R(0)=y-F(x)>0$ and $\lim\limits_{z\uparrow\left(\frac{x-x_\infty}{\alpha} \wedge y\right)}R(z)<0$. Since $z\mapsto R(z)$ is strictly decreasing (by strict monotonicity of $F$) it follows that there exists a unique solution to \eqref{optimalExercise}.

Finally, \eqref{eq54} follows by noticing that $0$ solves \eqref{optimalExercise} when $y=F(x)$ and by uniqueness of the solution. Analogously, \eqref{eq55} follows by noticing that $\frac{x-x_0}{\alpha}$ uniquely solves \eqref{optimalExercise}, since $F(x_0)=0$.
\vspace{0.35cm}

\emph{Proof of Lemma \ref{sigma}.}
\vspace{0.25cm}

The first equality in \eqref{eq38} follows from \eqref{eq10}. In order to prove the last inequality in \eqref{eq38}, we find by Lemma \ref{Properties}-(2) that
\begin{align}\label{eq36}
\frac{\sigma^2}{2}\psi^{(k+2)}(x;a,\sigma)+(a-bx)\psi^{(k+1)}(x;a,\sigma)-(\rho+kb)\psi^{(k)}(x;a,\sigma)=0.
\end{align} 
From \eqref{eq36}, recalling that $\psi^{(k+1)}>0$, we obtain
$$(a-bx)=-\frac{\sigma^2\psi^{(k+2)}(x;a,\sigma)}{2\psi^{(k+1)}(x;a,\sigma)}+(\rho+kb)\frac{\psi^{(k)}(x;a,\sigma)}{\psi^{(k+1)}(x;a,\sigma)}.$$
and we thus have 
\begin{align*}
&(a-bx)\Big[\psi^{(k+1)}(x;a,\sigma)^2-\psi^{(k)}(x;a,\sigma)\psi^{(k+2)}(x;a,\sigma)\Big]+b\psi^{(k+1)}(x;a,\sigma)\psi^{(k)}(x;a,\sigma)\\
=&(\rho+(k+1)b)\psi^{(k)}(x;a,\sigma)\psi^{(k+1)}(x;a,\sigma)-(\rho+kb)\psi^{(k)}(x;a,\sigma)^2\frac{\psi^{(k+2)}(x;a,\sigma)}{\psi^{(k+1)}(x;a,\sigma)}\\
& \hspace{0.15cm} + \underbrace{\frac{\sigma^2\psi^{(k+2)}(x;a,\sigma)}{2\psi^{(k+1)}(x;a,\sigma)}\Big[\psi^{(k)}(x;a,\sigma)\psi^{(k+2)}(x,a,\sigma)-\psi^{(k+1)}(x;a,\sigma)^2\Big]}_{>0\,\text{by Lemma \ref{Properties}}}\\
>&\frac{\psi^{(k)}(x;a,\sigma)}{\psi^{(k+1)}(x;a,\sigma)}\Big[(\rho+(k+1)b)\psi^{(k+1)}(x;a,\sigma)^2-(\rho+kb)\psi^{(k)}(x,a,\sigma)\psi^{(k+2)}(x;a,\sigma)\Big].
\end{align*}

We now aim at establishing that the last term on the right-hand side of the latter equation is positive. With regard to \eqref{eq38}, this would clearly imply that $\frac{\partial(\psi^{(k)}(x;a,\sigma)/\psi^{(k+1)}(x;a,\sigma))}{\partial\sigma}>0$. From \eqref{eq36} we have $$(\rho+(k+1)b)\psi^{(k+1)}(x;a,\sigma)=\frac{\sigma^2}{2}\psi^{(k+3)}(x;a,\sigma)+(a-bx)\psi^{(k+2)}(x;a,\sigma),$$ which then yields
\begin{align*}
&\frac{\psi^{(k)}(x;a,\sigma)}{\psi^{(k+1)}(x;a,\sigma)}\Big[(\rho+(k+1)b)\psi^{(k+1)}(x;a,\sigma)^2-(\rho+kb)\psi^{(k)}(x;a,\sigma)\psi^{(k+2)}(x;a,\sigma)\Big]\\
=&\frac{\psi^{(k)}(x;\sigma)}{\psi^{(k+1)}(x;a,\sigma)}\Big[\frac{\sigma^2}{2}\psi^{(k+3)}(x;a,\sigma)\psi^{(k+1)}(x;a,\sigma)\\
&+\psi^{(k+2)}(x;a,\sigma)\big((a-bx)\psi^{(k+1)}(x;a,\sigma)-(\rho+kb)\psi^{(k)}(x;a,\sigma)\big)\Big]\\
=&\frac{\sigma^2}{2}\frac{\psi^{(k)}(x;\sigma)}{\psi^{(k+1)}(x;a,\sigma)}\Big[\psi^{(k+3)}(x;a,\sigma)\psi^{(k+1)}(x;a,\sigma)-\psi^{(k+2)}(x;a,\sigma)^2\Big] > 0,
\end{align*}
where the last equality follows again by an application of \eqref{eq36}, and the last inequality by Lemma \ref{Properties}. Hence $\frac{\partial(\psi^{(k)}(x;a,\sigma)/\psi^{(k+1)}(x;a,\sigma))}{\partial\sigma}>0$ and the proof is completed.

\section{An Auxiliary Result}	
\label{App:B}

\renewcommand{\theequation}{B-\arabic{equation}}

\begin{lemma}\label{Propx_0xbar}
	Let $x_0$ be the solution to \eqref{def:x0} and 
	\begin{equation}
	\label{barx}
	\bar{x}:=\frac{a + \rho c}{\rho + b}.
	\end{equation}
	 We have $$\bar{x} < x_0.$$
\end{lemma}

\begin{proof}
	Define $H(x):=(x-c)\psi^\prime(x)-\psi(x),$ $x\in\mathbb{R}$. Since $\psi$ satisfies
	$$\frac{\sigma^2}{2}\psi''(x)+(a-bx)\psi'(x)-\rho\psi(x)=0,\quad\text{for all $x\in\mathbb{R}$},$$ 
	and $\frac{\sigma^2}{2}\psi''(x)>0$, we find $-\psi(x)<-\frac{(a-bx)}{\rho}\psi^\prime(x),\,\forall x\in\mathbb{R}.$ 
	Thus, we have 
	$$H(\bar{x})<(\bar{x}-c)\psi'(\bar{x})-\frac{(a-b\bar{x})}{\rho}\psi'(\bar{x})=\Big[(\bar{x}-c)\rho-(a-b\bar{x})\Big]\frac{\psi^\prime(\bar{x})}{\rho}=0,$$ 
	by the definition of $\bar{x}$. Since $H(x_0)=0$, $H(x)<0$ for all $x<x_0$ and $H(x)>0$ for all $x>x_0$, it must necessarily be $\bar{x}< x_0$.
\end{proof}

%%%%%%%%%%%%%%%%%%%%%%%%%%%%%%%%%%%%%%%%%%%%%%%%%%%%%%%%%%%%%%%%%%%%%%%%%%%%%%%%%%%%%%%%%%%%%%%%%%%%%%%%%%%%%%%%%%%%%%%%%%%%%%%%%%%%%%%%%%%%%%%

\medskip

\indent \textbf{Acknowledgments.}. Financial support by the German Research Foundation (DFG) through the Collaborative Research Centre 1283 ``Taming uncertainty and profiting from randomness and low regularity in analysis, stochastics and their applications'' is gratefully acknowledged by the authors. We thank Stefan Ankirchner, Dirk Becherer, Todor Bilarev, Ralf Korn, Frank Riedel, Wolfgang J.\ Runggaldier, and Thorsten Upmann for valuable discussions and comments. In particular, we are thankful to Peter Frentrup for pointing out a mistake in a previous version of this manuscript.

%%%%%%%%%%%%%%%%%%%%%%%%%%%%%%%%%%%%%%%%%%%%%%%%%%%%%%%%%%%%%%%%%%%%%%%%%%%%%%%%%%%%%%%%%%%%%%

\end{document}